\documentclass[a4paper, 11pt]{article}

\usepackage{./standard, ./thm, ./macros}

\begin{document}

\title{Malcev Completions, Hodge Theory, and Motives}
\author{Emil Jacobsen\footnote{%
Part of this work was supported by
the Swiss National Science Foundation (SNSF), project 178729,
and part of it was supported by Dan Petersen's Wallenberg Scholar fellowship.}}
\maketitle

\bigskip

\begin{abstract}
We prove that, on a smooth, connected variety in characteristic zero admitting a rational point,
local systems of geometric origin
are stable under extension in the category of all local systems.
As a consequence of this,
we obtain a (Nori) motivic strengthening of Hain's theorem
on Malcev completions of monodromy representations.

Our methods are Tannakian,
and rely on an abstract criterion for ``Malcev completeness'',
which is proved in the first part of the paper.
A couple of secondary applications of this criterion are given:
an alternative proof of D'Addezio--Esnault's theorem,
which says that local systems of Hodge origin are stable under extension
in the category of all local systems;
a generalization of the theorem of Hain, mentioned above,
which also affirms a conjecture of Arapura;
and an alternative proof of a theorem of Lazda,
which under suitable assumptions
gives an isomorphism between the relative unipotent de Rham fundamental group
and the unipotent de Rham fundamental group of the special fibre.
\end{abstract}

\bigskip

\tableofcontents

\section{Introduction}
The main result of this paper can be stated as:
local systems of geometric origin are stable under extension.
Let us have a closer look at what is meant by this.

Fix a subfield $k$ of $\C$.
For a $k$-variety $X$,
Ivorra and Morel (\cite{ivorra-morel}) define the category of \emph{perverse motives} on $X$,
here denoted by $\PM(X)$.
It is a $\Q$-linear abelian category and comes with a faithful exact
\emph{Betti realization functor}
$\PM(X) \to \clg{P}(X)$ landing in the category $\clg{P}(X)$ of perverse sheaves (with rational coefficients) on $X$.
Over the base field $k$,
perverse motives are the same as Nori motives.

Now let $X$ be a smooth, connected $k$-variety.
Then, the category of local systems (with rational coefficients) $\Loc(X)$ on $X$
forms an abelian subcategory of $\clg{P}(X)$ (after shifting).
This lets us define the category of \emph{motivic local systems}, $\MLS(X)$,
as the full subcategory of perverse motives on $X$ generated by those whose Betti realizations are (shifted) local systems.
We get an induced Betti realization functor $b_X^*\colon \MLS(X) \to \Loc(X)$.

\begin{rmk}
Apart from the results of Ivorra--Morel,
the main reason that motivic local systems are well-behaved enough for our purposes
is the work of Terenzi in~\cite{luca},
to which we return later in the introduction.
\end{rmk}

\begin{defn}
A local system on $X$ is \emph{of geometric origin}, or \emph{comes from geometry},
if it is a subquotient of one in the image of the Betti realization $b_X^*\colon \MLS(X) \to \Loc(X)$.
\end{defn}

Our main result can then be stated as follows:
\begin{thm}[\nref{Thm.}{thm:motivic_hain}.1]\label{thm:easy_main_thm}
Let $X$ be a smooth and connected $k$-variety with a rational point.
Then local systems of geometric origin
are stable under extension in $\Loc(X)$.
In other words,
if
\begin{align*}
0 \to L' \to L \to L'' \to 0
\end{align*}
is a short exact sequence in $\Loc(X)$ such that $L'$ and $L''$ are of geometric origin,
then so is $L$.
\end{thm}

\begin{rmk}
There are multiple other definitions of local systems of geometric origin in the literature.
For some recent examples, see~\cite[Def.~1.6.1]{ayoub_anabelian} and~\cite[685]{landesman-litt}.
Further back,
Simson gives a vague definition in~\cite[9]{simpson}
(using the terminology of a local system being ``motivic'', rather than ``of geometric origin''),
and we believe our definition is in the spirit of this one.
As a consequence of our main theorem (Thm.~\ref{thm:easy_main_thm}),
it is likely our definition of geometric origin agrees with that of Ayoub.
\end{rmk}

The introduction will continue as follows.
First,
in Section~\ref{sec:intro_motivic_malcev},
we introduce the notion of Malcev completions.
This lets us
state our next result,
which serves as a motivic version of a theorem of Hain.
After that, we quickly introduce the motivic fundamental group
in Section~\ref{sec:intro_fund_gp} and give an alternative formulation of the Thm.~\ref{thm:easy_main_thm}.
The key technical input in our proof of the main theorem,
is a criterion for ``Malcev completeness'',
which will be discussed towards the end of the introduction
(Section~\ref{sec:intro_malcev_criterion}).
Before we get there,
we will go over two other applications of this criterion (Section~\ref{sec:intro_other}),
as well as some background which places our main results in a larger context (Section~\ref{sec:background}).
A brief sketch of the proofs of our main results closes out the introduction in Section~\ref{sec:intro_sketch_pfs}.

\subsection{Motivic Malcev completions}\label{sec:intro_motivic_malcev}
Keep $k$ as above and let $X$ be a smooth, connected $k$-variety
with a rational point $x$.
Recall that $x^*\colon \Loc(X) \to \VecOp$
induces an equivalence $\Loc(X) \simeq \Rep \pi_1(X,x)$,
where the latter denotes the category of finite-dimensional $\Q$-linear
representations of the fundamental group of $X^\rom{an}$ pointed at $x$.

Before we go further,
we need to introduce the notion of Malcev completion.
Suppose $P \to K$ is a homomorphism of affine group schemes over a field.
Then, the \emph{Malcev completion} of this homomorphism
satisfies the following universal property:
it is the initial group $\clg{G}$ equipped with
a morphism from $P$ and a morphism to $K$,
such that the latter has pro-unipotent kernel
and such that everything commutes
(see \nref{Section}{sec:malcev} for a much more thorough treatment):
\begin{equation*}\begin{tikzcd}
1 \ar[r] & \clg{U} \ar[r] & \clg{G} \ar[r]     & K \\[-0.5em]
         &                & P \ar[u] \ar[ur]
\end{tikzcd}\end{equation*}
If $P$ is an abstract group with a homomorphism to the rational points of $K$,
we can still talk about the Malcev completion,
by first replacing $P$ with its pro-algebraic completion.

Now let $M \in \MLS(X)$ be a motivic local system on $X$.
The underlying local system $b_X^*M$
corresponds to a (monodromy) representation
$\varrho \colon \pi_1(X, x) \to \GL(V)$.
Let $\hat\pi_1^M(X,x)$ be the Malcev completion of $\varrho$.

We're able to obtain the next theorem as a corollary of the previous one,
using basic properties of Malcev completions worked out in Section~\ref{sec:malcev}.
It's a motivic version of Hain's theorem~\cite[Thm.~13.1]{hain}.

\begin{thm}[{\nref{Thm.}{thm:motivic_hain}.2}]\label{thm:intro_motivic_hain}
We have that $\hat\pi_1^M(X,x)$
is naturally a Nori motive.
\end{thm}

\begin{rmk}
An affine $\Q$-group scheme $G$ \emph{is a Nori motive}
if $\clg{O}(G)$ is a Hopf algebra object in the ind-category of Nori motives.
\end{rmk}

\begin{rmk}
The special case of the unipotent completion of $\pi_1(X,x)$
was treated by Cushman; see~\cite[Thm.~3.1]{cushman}.
\end{rmk}

\subsection{The motivic fundamental group}\label{sec:intro_fund_gp}
We keep $k$, $X$, and $x$ as before.
Let $\Loc_\rom{geo}(X)$ denote the full subcategory of local systems of geometric origin in $\Loc(X)$.
Note that $\Loc_\rom{geo}(X)$ is neutral Tannakian,
neutralized, for instance, by (the restriction of) $x^*$.
The Tannaka dual group of $\Loc_\rom{geo}(X)$ is denoted by
$\pi_1^\rom{mot}(X,x)$
and is called the \emph{motivic fundamental group} of $X$ (pointed at $x$).
By construction,
there is a homomorphism $\pi_1(X,x) \to \pi_1^\rom{mot}(X,x)$.
Theorem~\ref{thm:easy_main_thm} is equivalent to the following:

\begin{thm}[Thm.~\ref{thm:motivic_hain}.1]\label{thm:intro_motivic_malcev_complete}
The homomorphism $\pi_1(X,x) \to \pi_1^\rom{mot}(X,x)$
is \emph{Malcev complete},
i.e.,
its Malcev completion is isomorphic to $\pi_1^\rom{mot}(X,x)$
(along the canonical map).
\end{thm}

This can be viewed as a universal version of Theorem~\ref{thm:intro_motivic_hain}.

\subsection{Other results}\label{sec:intro_other}
The proof of the main theorem relies on a criterion for Malcev completeness
(which is discussed in \nref{Section}{sec:intro_malcev_criterion}).
In this subsection,
I will briefly discuss two other applications of this criterion.

\subsubsection{Hodge theory}
Let $X$ be a smooth, geometrically connected variety over a subfield $k$ of $\C$.
Fix a $\C$-point $x$ of $X$,
and let $\pi_1^\rom{Hdg}(X,x)$ denote the \emph{Hodge fundamental group} of $X$,
i.e.,
the Tannaka dual of the category of
admissible variations of mixed Hodge structures
on $X$,
with respect to the fibre functor induced by $x$
(see \nref{Section}{subsec:hodge}).
\begin{thm}[\nref{Thm.}{thm:generalized_hain}]\label{thm:hain_intro}\
\begin{enumerate}
\item (\cite[Thm.~4.4]{daddezio-esnault})
The map $\pi_1(X,x) \to \pi_1^\rom{Hdg}(X,x)$ is Malcev complete.
Equivalently,
the full subcategory of local systems on $X$ of Hodge origin is
stable under extension in the category of local systems on $X$.
In particular,
every unipotent local system on $X$ is of Hodge origin.
\item
Given an admissible variation of mixed Hodge structures $V$ on $X$,
the Malcev completion of the monodromy representation $\pi_1(X,x) \to \GL(V_x)$
carries a canonical (rational) mixed Hodge structure.
\end{enumerate}
\end{thm}

\begin{rmk}\label{rmk:intro_carry_MHS}
An affine $\Q$-group scheme $G$ \emph{carries a mixed Hodge structure}
if $\clg{O}(G)$ is a Hopf algebra object in the ind-category of rational, graded-polarisable, mixed Hodge structures.
\end{rmk}

The Malcev completeness of $\pi_1(X,x)\to\pi_1^\rom{Hdg}(X,x)$
has already been proved by D'Addezio and Esnault by a different method
in~\cite[Thm.~4.4]{daddezio-esnault}.
The second part of the theorem
did not appear in~\cite{daddezio-esnault},
and is a generalization of Hain's theorem
(\cite[Thm.~13.1]{hain}, see \nref{Section}{sec:background_hain}).

\subsubsection{De Rham fundamental groups}
Let $X \to S$ be a nice enough morphism of nice enough $k$-schemes
(see \nref{Section}{sec:deRham} for the precise statements).
In~\cite{lazda}, Lazda introduces the relative (unipotent) de Rham fundamental group
$\pi_1^\rom{dR}(X/S)$,
and shows that it's isomorphic,
via a natural map
to the (unipotent) de Rham fundamental group
$\pi_1^\rom{dR}(X_s)$
of the special fibre.
In \nref{Section}{sec:deRham},
Lazda's theorem is deduced using the unipotent isomorphism criterion
(see \nref{Thm.}{thm:U_intro} below).

\subsection{Background}\label{sec:background}
Let us spend some time putting the main theorem in a historical context.

\subsubsection{The fundamental short exact sequence}\label{sec:background_SES}
Fix a field $k$ and an algebraic closure $\bar{k}$.
The \'etale fundamental group $\pi_1^\et(X,x)$ of a connected $k$-pointed $k$-variety $(X,x)$
was introduced by Grothendieck in~\cite{SGA1}.
By functoriality,
the structure map $X \to \Spec(k)$
and the point $x$ induce a split surjection
\begin{equation}\begin{tikzcd}\label{eq:et_split_surj}
\pi_1^\et(X,x) \ar[r, two heads] & \Gal(\bar{k}/k)
  \ar[l, bend right, start anchor = north west, end anchor = north east, yshift=-.5em]
.
\end{tikzcd}\end{equation}
Under suitable assumptions,
Grothendieck gives a precise description of the kernel of this surjection.
Namely, it is given by the geometric \'etale fundamental group
$\pi_1^\et(\bar{X},\bar{x})$
of $X$,
where $\bar{X} := X \otimes_k \bar{k}$.
Thus, there is a split exact sequence of profinite groups:
\begin{equation}\begin{tikzcd}\label{eq:et_fund_seq}
1 \ar[r] & \pi_1^\et(\bar{X},\bar{x}) \ar[r] & \pi_1^\et(X,x) \ar[r] & \Gal(\bar{k}/k)
  \ar[l, bend right, start anchor = north west, end anchor = north east, yshift=-.5em]
  \ar[r] & 1.
\end{tikzcd}\end{equation}
From now on,
we suppose that $k$ is a subfield of the complex numbers.
We then have access to the topological fundamental group
$\pi_1(X,x)$ of the complex analytification $X^\rom{an}$ of $X$.
There is a natural map
\begin{align}\label{eq:top_to_et}
\pi_1(X,x) \to \pi_1^\et(\bar{X},\bar{x}),
\end{align}
and Grothendieck also tells us exactly what it is,
namely a profinite completion.
I am interested in \emph{motivic analogues} of the above results.
More precisely,
\begin{enumerate}[(a)]
\item\label{enum:goal_splitting}
a motivic analogue of the split surjection~\heqref{eq:et_split_surj},
\item\label{enum:goal_kernel}
a description of its kernel, and
\item\label{enum:goal_relationship}
some understanding of the relationship between this kernel and the topological fundamental group.
\end{enumerate}

The motivic analogue of the absolute Galois group $\Gal(\bar{k}/k)$ of $k$
is given by the \emph{motivic Galois group},
here denoted by $\mathscr{G}^\rom{mot}(k)$.
This affine $\Q$-group scheme can be defined as the Tannaka dual of the category of Nori motives,
or equivalently as the spectrum of a certain \emph{motivic Hopf algebra}
constructed by Ayoub in~\cite{motivic_hopf_I,motivic_hopf_II}
(see~\cite{choudhury-gallauer} for the comparison).
By equipping the derived category of perverse motives with a well-behaved tensor product,
Terenzi shows in~\cite[Thm.~6.2]{luca} that $\MLS(X)$ is neutral Tannakian.
We denote its Tannaka dual by $G(X,x)$:
this is the motivic analogue of $\pi_1^\et(X,x)$.

The structure map of $X$ and the rational point $x$ induce a split surjection
$G(X,x) \to \scr{G}^\rom{mot}(k)$,
by way of the inverse image functors on perverse motives (see~\cite{ivorra-morel}).
This is our motivic analogue of~\heqref{eq:et_split_surj},
in line with goal~\heqref{enum:goal_splitting}.
Denote the kernel by $K(X,x)$.
By definition,
we have a homomorphism $\pi_1(X,x) \to K(X,x)$
and
(the Zariski closure of)
its image
is the motivic fundamental group $\pi_1^\rom{mot}(X,x)$.
Our main theorem
says that $\pi_1(X,x) \to K(X,x)$ is Malcev complete.
This implies firstly that $K(X,x) = \pi_1^\rom{mot}(X,x)$,
which we may view as a result in the direction of goal~\heqref{enum:goal_kernel}.
Then, the fact that $\pi_1(X,x) \to \pi_1^\rom{mot}(X,x)$ is Malcev complete,
is a result in the direction of goal~\heqref{enum:goal_relationship}.

In conclusion,
we have a split exact sequence of affine group schemes over $\Q$
\begin{equation}\begin{tikzcd}
1 \ar[r] & \pi_1^\rom{mot}(X,x) \ar[r] & G(X,x) \ar[r] & \scr{G}^\rom{mot}(k)
  \ar[l, bend right, start anchor = north west, end anchor = north east, yshift=-.5em]
  \ar[r] & 1,
\end{tikzcd}\end{equation}
and a homomorphism $\pi_1(X,x) \to \pi_1^\rom{mot}(X,x)$,
which is Malcev complete.

\subsubsection{Motivic structures on fundamental groups}\label{sec:background_hain}
The singular homology functors $H_i(\blank,\Q)$ on algebraic varieties over $\C$
factor through the category of mixed Hodge structures,
by the seminal work of Deligne.
Conjecturally,
they factor further,
through a (Tannakian) category of motives,
enjoying many useful properties.
By Hurewicz,
the functor $\pi_1(\blank,\blank)^\rom{ab} \otimes \Q$ on pointed algebraic $\C$-varieties
is isomorphic to $H_1(\blank,\Q)$,
so it too should factor through motives.
Put differently,
the abelianization of the fundamental group of an algebraic variety is canonically a motive.
It's natural to investigate whether a larger part of the fundamental group is a motive,
as passing to the abelianization is rather destructive.
In particular,
one can ask to what extent Hodge theory can be extended to the fundamental group,
i.e., what part of the fundamental group carries a natural mixed Hodge structure.

In this direction,
and improving upon~\cite{morgan},
Hain showed in~\cite{hain_unip_1} that
the unipotent completion of the fundamental group carries a natural mixed Hodge structure
(see Rmk.~\ref{rmk:intro_carry_MHS}).
Hain then went further in~\cite{hain}
and proved the following:
\begin{thm}[{\cite[Thm.~13.1]{hain}}]
Let $X$ be a smooth connected variety over $\C$,
and let $x$ be a complex point of $X$.
Let $V$ be a polarizable variation of Hodge structure on $X$,
such that the image of the monodromy representation $\pi_1(X,x) \to \GL(V_x)$
is Zariski dense in $\Aut(V_x,\langle,\rangle)$.
Then the Malcev completion of this monodromy representation
carries a canonical (real) mixed Hodge structure.
\end{thm}

Motivic strengthenings of Hain's unipotent theorem have been established in some cases.
For instance,
Deligne (\cite{deligne-P1-3pts}) equipped the unipotent completion of the fundamental group of a rational variety
with the structure of a compatible system of realizations, which is not too far from that of a motive.
This was strengthened in~\cite{deligne-goncharov},
where the unipotent completion of the fundamental group of a unirational variety defined over a number field
was equipped with the structure of a mixed Tate motive,
in Voevodsky's sense.

A motivic strengthening of Hain's theorem on Malcev completions is achieved in
\nref{Thm.}{thm:intro_motivic_hain}.
Moreover,
the second part of the Hodge theoretic result \nref{Thm.}{thm:hain_intro}
generalizes Hain's theorem,
by allowing arbitrary admissible variations.

\subsection{The Malcev completeness criterion}\label{sec:intro_malcev_criterion}
The proofs of these three theorems are based on a general criterion
for a homomorphism to be Malcev complete.
All the groups we consider are affine group schemes over a fixed field.
Let $P \to K$ be a group homomorphism,
and denote the restriction functor of (finite-dimensional) $K$-representations along this morphism by
$\Res^K_P\colon \Rep K \to \Rep P$.
It's an easy observation that
$P\to K$ is Malcev complete if and only if
\begin{itemize}
\item[(i)] $P \to K$ is surjective, and
\item[(ii)] the image of $\Res^K_P$ is stable under extension in $\Rep P$.
\end{itemize}

Let $K \twoheadrightarrow S$ be a quotient by a pro-unipotent normal subgroup.
We write $\Gamma(K,\blank)$ and $\Gamma(P,\blank)$ for $K$- and $P$-invariants respectively.
It is then not difficult to see that Malcev completeness of $P \to K$ is equivalent to
\begin{itemize}
\item[(iii)] the composition $P \to K \to S$ is surjective, and
\item[(iv)] the natural maps
\begin{align*}
R^i\Gamma(K,\blank) \to R^i\Gamma(P,\blank) \circ \Res^K_P,
\end{align*}
are isomorphisms for $i=0,1$.
\end{itemize}
In applications, the following theorem offers a condition which is easier to check than~(iv):
\begin{thm}[Malcev completeness criterion, Thm.~\ref{thm:M}]\label{thm:M_intro}
Let $G = K \rtimes H$ be a semi-direct product
and consider a homomorphism $P \to K$.
Let $K \twoheadrightarrow S$ be quotient by a pro-unipotent normal subgroup.
Then $P\to K$ is Malcev complete if
$P \to S$ is surjective and the natural maps
\begin{align*}
R^i\Gamma(K,\blank) \circ \Res^G_K \to R^i\Gamma(P,\blank) \circ \Res^G_P
\end{align*}
between functors $\Rep G \to \VecOp$ are isomorphisms for $i=0,1$.
\end{thm}
In other words,
when checking condition (iv) above,
it's enough to work with objects in the image of $\Res^G_K$.
Note that the group $H$ and its action on $K$ can be chosen completely arbitrary,
and therein lies the strength of the theorem.
The proof of \nref{Thm.}{thm:M_intro} works by reducing to the following special case:
\begin{thm}[Unipotent isomorphism criterion, Thm.~\ref{thm:U2}]\label{thm:U_intro}
Let $G = K \rtimes H$ be a semi-direct product
and consider a homomorphism $P \to K$.
Then $P\to K$ is an isomorphism if
$P$ is pro-unipotent and the natural maps
\begin{align*}
R^i\Gamma(K,\blank) \circ \Res^G_K \to R^i\Gamma(P,\blank) \circ \Res^G_P
\end{align*}
between functors $\Rep G \to \VecOp$ are isomorphisms for $i=0,1$.
\end{thm}
\begin{rmk}
Using Tannaka duality, it's easy to see that,
when $P$ is pro-unipotent,
$P \to K$ is an isomorphism if and only if condition~(iv) above holds.
Moreover, the assumptions in \nref{Thm.}{thm:U_intro} easily imply that $K$ too is pro-unipotent,
so that $S$ is trivial and the surjectivity of $K \to S$ is automatic,
really making it a special case of \nref{Thm.}{thm:M_intro}.
\end{rmk}

The following direct corollary to Thm.~\ref{thm:M_intro} has been suggested by Marco D'Addezio:
\begin{cor}
Suppose $\charac F = 0$.
Let $P \to G$ be a homomorphism
whose image is contained in the pro-unipotent radical $R_\rom{u}G$,
and suppose that $P$ is pro-unipotent.
Then $P \to R_\rom{u}G$ is an isomorphism if and only if
the natural maps
\begin{align*}
\Res^{G_\rom{red}}_1 \circ\, R^i\Ind_G^{G_\rom{red}} \to R^i\Gamma(P,\blank) \circ \Res^G_P
\end{align*}
are isomorphisms for $i=0,1$.
\end{cor}

\begin{rmk}
Keep the setup as in \nref{Thm.}{thm:M_intro}.
The theorem takes three assumptions:
the cohomological assumptions for $i=0,1$,
and the surjectivity of $P \to S$.
The result is the Malcev completeness of $P \to K$,
which essentially has two parts:
the surjectivity of $P \to K$ (listed as (i) above),
and condition (ii) on stability under extension.
If one only cares about the first consequence,
it is shown in~\cite{marco_private} that the $i=1$ condition isn't needed.
More precisely,
D'Addezio--Esnault
show in ibid.\ how to use results from~\cite[Appendix~A]{daddezio-esnault}
to lift surjectivity of $P \to S$ to surjectivity of $P \to K$
when $\Gamma(K,\blank) \circ \Res^G_K \to \Gamma(P,\blank) \circ \Res^G_P$ is an isomorphism.
\end{rmk}

\subsection{A sketch of the proofs}\label{sec:intro_sketch_pfs}
Keep the notation as in the unipotent isomorphism criterion (\nref{Thm.}{thm:U_intro}).
The choice of a section $H \to G$ makes $\clg{O}(K)$ into an algebra object in $\Ind\Rep G$.
To prove \nref{Thm.}{thm:U_intro}, we need to show that $P \to K$ is an isomorphism,
or equivalently that $\clg{O}(K) \to \clg{O}(P)$ is an isomorphism.
We do this by constructing
$\clg{O}(P)$ and $\clg{O}(K)$
as explicit colimits
in $\Ind\Rep P$ and $\Ind\Rep G$, respectively,
and then showing that the natural map between them in $\Ind\Rep P$ is an isomorphism.
They are constructed as certain universal iterated extensions starting with the trivial object $\bm 1$,
and this is carried out in Sections~\ref{subsec:U} and~\ref{subsec:W}.
The rest of \nref{Section}{sec:thmU} is then dedicated to showing that
these constructions work as intended,
and finally deducing the result.

This construction of the regular representation of a pro-unipotent group
as a universal iterated extension has appeared in the literature in several special cases,
for example in \cite{hadian,AIK,lazda,log-de-Rham-arx}.
As far as we know,
the general form of \nref{Thm.}{thm:U_intro} we prove has not appeared before,
though it might have been known to experts.

To prove the Malcev completeness criterion (\nref{Thm.}{thm:M_intro}),
we reduce to the unipotent case.
The details of this reduction are a bit technical,
but the idea is simple:
$K$ and the Malcev completion of $P \to K$
are by assumption both extensions
of $S$ by pro-unipotent groups.
With this in mind we reduce to showing that these pro-unipotent groups are isomorphic,
and finish the job using \nref{Thm.}{thm:U_intro}.
The heart of this reduction argument is \nref{Prop.}{prop:module},
which might also be of independent interest.

In order for \nref{Thm.}{thm:M_intro} to be applicable in the Hodge-theoretic context of \nref{Thm.}{thm:hain_intro},
the assumptions need to be checked,
but let us first provide a dictionary between the groups appearing in the two theorems.
The role of $P$ is played by (the pro-algebraic completion of) $\pi_1(X,x)$,
while that of $G$ is played by the Tannaka dual of the category of admissible variations on $X$.
The group $H$ is then taken to be the \emph{Hodge group},
i.e., the Tannaka dual of the category of rational mixed Hodge structures.
Finally, $K$ is in this context the kernel of the natural homomorphism\footnote{%
  It's induced under Tannaka duality by
  the functor sending a mixed Hodge structure to
  the associated constant variation.}
$G \to H$,
and the map from $P \to G$ is easily seen to factor through $K$.
The group-cohomological assumption is then established using the theory of mixed Hodge modules,
in \nref{Prop.}{prop:enriched_coh_assump} and \nref{Lemma}{lem:MHM_ambient},
and the surjectivity assumption is established using Prop.~\ref{prop:enriched_surjectivity} and Deligne's semi-simplicity theorem.
This gives the Malcev completeness of the relevant map,
and the second part follows as a corollary,
see \nref{Prop.}{prop:generalized_hain_enriched}.

In the motivic context of \nref{Thm.}{thm:easy_main_thm},
we use results from~\cite{ivorra-morel} on perverse motives
to check the cohomological assumption of \nref{Thm.}{thm:M_intro} (see~\ref{lem:motivic_ambient}).
The surjectivity assumption again boils down to Deligne's semi-simplicity theorem,
as well as a well-behaved theory of weights on perverse motives.

\subsection{Acknowledgements}
I would like first and foremost to thank my PhD advisor Joseph Ayoub,
without whom this project wouldn't even exist.
Secondly, I'm grateful to Andrew Kresch, Luca Terenzi and Dan Petersen for many useful discussions,
and the algebraic geometry group at the University of Z\"urich for the extremely pleasant and motivating working environment
I benefited from when the bulk of this work was done.
Lastly, I want to thank
Marco D'Addezio
and two anonymous referees
for many useful comments on previous versions of this text.

\newpage
\part{Affine group schemes over a field}\label{part:abstract}
\bigskip\bigskip

\section{Preliminaries}
We fix a field $F$ and assume all additive categories to be $F$-linear.
The category of finite-dimensional $F$-vector spaces will be denoted by $\VecOp$.
We write $\bm1$ for the unit object in a monoidal category.
When there are several monoidal categories around,
it will be understood from context which unit object we mean.

By \emph{pro-algebraic group} we mean an affine group scheme over $F$.

We will use the term \emph{tensor category} to mean an $F$-linear symmetric monoidal category.
That is, an $F$-linear category $\clg{C}$ equipped with an $F$-bilinear functor
$(\blank\otimes\blank)\colon \clg{C}\times\clg{C}\to\clg{C}$,
a unit object $\bm1$,
as well as compatible natural isomorphisms encoding associativity, commutativity,
and the unit relations,
that satisfy the usual axioms.
We will say a functor $f$ between tensor categories
is \emph{monoidal} when it's equipped with an isomorphism $f(\bm1) \simeq \bm1$
and a natural isomorphism $f(X\otimes Y) \simeq f(X)\otimes f(Y)$,
both compatible with the associator and commutator isomorphisms of the two categories.
Finally, a natural transformation $\phi\colon f\to g$ between monoidal functors $f$ and $g$
is \emph{monoidal} when it's compatible with the monoidal structures of $f$ and $g$.

An object $X$ in a tensor category $\clg{C}$ is \emph{(strongly) dualizable}
if there exists an object $X^\vee$ (called a \emph{dual} of $X$)
and maps
$\ev\colon X\otimes X^\vee \to \bm1$,
$\coev\colon \bm1 \to X^\vee \otimes X$
that turn
$X^\vee\otimes(\blank)$
into a right adjoint of
$X\otimes(\blank)$.
Moreover, $\clg{C}$ is \emph{rigid} if every object in $\clg{C}$ is dualizable.
Duality is reflexive, in the sense that $X$ is a dual of $X^\vee$.
Monoidal functors preserve duals.

\subsection{Unipotency}\label{sec:unipotency}
\begin{defn}
Let $\clg{A}$ be an abelian category and $f\colon\clg{C}\to\clg{A}$ a functor.
Say an object $U$ in $\clg{A}$ is \emph{$f$-unipotent} if
it can be built in a finite number of steps from the essential image of $f$,
using extensions.
More precisely, if there exists a filtration
\begin{align*}
0=U_{-1}\subseteq U_0\subseteq\dotsc\subseteq U_n=U,
\end{align*}
such that the quotients $U_i/U_{i-1}$ are in the essential image of $f$.
We call such a filtration an \emph{$f$-unipotent filtration} of $U$, of \emph{length} $n$.
If every object of $\clg{A}$ is $f$-unipotent,
then $\clg{A}$ is said to be \emph{$f$-unipotent}
(or sometimes \emph{relatively unipotent over $\clg{C}$}).
\end{defn}

\begin{defn}
Let $U$ be an $f$-unipotent object, as above.
Let $n$ be the smallest integer such that $U$ admits an $f$-unipotent filtration of length $n$.
We say the \emph{$f$-unipotency index} of $U$ is $n$.
\end{defn}

\begin{rmk}
The usefulness of unipotency is that it allows for \emph{unipotent induction},
that is,
induction on the unipotency index.
This technique lets us prove many statements
by checking that they hold for a smaller class of objects,
and that they behave well with respect to extension.
\end{rmk}

\begin{lem}\label{lem:unip-ext}
Unipotency is stable under extension.
\end{lem}
\begin{proof}
Let $f\colon \clg{C}\to\clg{A}$ be as above.
Consider an extension
\begin{align*}
0\to U\to A\to W\to 0,
\end{align*}
where $U$ and $W$ are both $f$-unipotent.
In the case when $W=f(C)$,
$A$ is unipotent by definition.
We now proceed by induction on the unipotency index of $W$.
Let
\begin{align*}
0\to W'\to W\to f(C)\to 0,
\end{align*}
such that $W'$ is of a lower unipotency index than $W$.
By induction, any extension of $W'$ by something $f$-unipotent is unipotent.
The pullback $A'$ of $A\to W$ and $W'\to W$ is therefore unipotent,
and $A$ is an extension
\begin{align*}
0\to A'\to A\to f(C)\to 0,
\end{align*}
by standard arguments. 
\end{proof}

\begin{defn}
With $f\colon \clg{C}\to\clg{A}$ as above,
we say that $\clg{A}$ is \emph{weakly $f$-unipotent}
if every non-zero object of $\clg{A}$ admits a non-zero subobject in the essential image of $f$.
\end{defn}

Unipotency immediately implies weak unipotency,
but the converse is only true under a finiteness condition.

\begin{prop}
Let $f\colon \clg{C}\to \clg{A}$ be as above.
\begin{enumerate}[(i)]
\item If $\clg{A}$ is $f$-unipotent,
then it is weakly $f$-unipotent.
\item If $\clg{A}$ is weakly $f$-unipotent,
then any object in $\clg{A}$ of finite length is $f$-unipotent.
\end{enumerate}
\end{prop}
\begin{proof}
(i) Let $A\in\clg{A}$ be non-zero.
Fix an $f$-unipotent filtration $A_i$ of $A$ as above.
Then the minimal $i$ such that $A_i$ is non-zero gives the desired subobject.

(ii) Let $A$ be non-zero of finite length.
We get a short exact sequence
\begin{align*}
0\to f(C)\to A\to A_1\to 0.
\end{align*}
Then $A_1$ has a non-zero subobject $f(C_1)$,
and we iterate this until,
by finiteness of the length of $A$,
we get $A_n=f(C_n)$.
Using that unipotency is closed under extension (\nref{Lemma}{lem:unip-ext}),
we conclude that the $A_i$ and finally $A$ are all $f$-unipotent.
\end{proof}

\begin{rmk}
When $\clg{A}$ is monoidal, $\clg{C}=\VecOp$, and $f\colon \VecOp\to \clg{A}$ is monoidal,
then $f$ is uniquely determined, up to monoidal natural isomorphism
(cf. \nref{Rmk.}{rmk:fibre-section}).
In this case, we usually just say that $\clg{A}$ is \emph{unipotent},
without reference to $f$.
\end{rmk}

\subsection{Neutral Tannakian categories}\label{sec:tannaka}
The notion of a (neutral) Tannakian category is central in the sequel.
A \emph{neutral Tannakian category} (over $F$) is
a rigid abelian tensor category $(\clg{A},\otimes,\bm1)$
such that $\End(\bm1)=F$,
and which admits a faithful, exact, monoidal functor $\clg{A}\to\VecOp$:
such a functor is called a \emph{(neutral) fibre functor}.
A \emph{neutralized Tannakian category} is a Tannakian category together with a neutral fibre functor.

\begin{rmk}\label{rmk:exact-tens-iHom}
The tensor product is required to be $F$-bilinear,
and it is automatically exact in both arguments (\cite[Prop. 1.17]{deligne-milne}).
The rigidity implies existence of an internal Hom functor,
which is (right and left) adjoint to the tensor product,
and which is hence also exact in both arguments.
We denote it by $\iHom(\blank,\blank)$.
\end{rmk}

By \emph{morphism} of neutral Tannakian categories $\clg{A}$ and $\clg{B}$
we shall mean an exact, faithful, monoidal functor.
We will always write such a functor decorated with a raised asterisk,
e.g.\ $p^*\colon \clg{A}\to \clg{B}$.
These functors always have ind-right adjoints,
i.e., functors $p_*\colon \Ind\clg{B}\to\Ind\clg{A}$,
for which we use the same letters but lowered asterisks.
These ind-right adjoints are left exact
and can always be right derived,
as the ind-category of a neutral Tannakian category has enough injectives
(see, e.g., \cite[Prop.~I.3.9]{jantzen}).
By \emph{morphism} of neutralized Tannakian categories $(\clg{A},f^*)$ and $(\clg{B},\varphi^*)$
we shall mean a morphism $p^*\colon\clg{A}\to\clg{B}$ together with a monoidal isomorphism $\varphi^*\circ p^* \simeq f^*$.

This gives 2-categories $\nTan$ and $\nTanpt$ of neutral and neutralized Tannakian categories, respectively,
after we define what the 2-morphisms are.
In $\nTan$, they are the monoidal, invertible natural transformations.
In $\nTanpt$, they are the same, with the additional requirement
that they be compatible with the extra data in the morphisms (the monoidal isomorphisms above).

\begin{rmk}[On sections of fibre functors.]\label{rmk:fibre-section}
A fibre functor $f^*$ of a Tannakian category $\clg{B}$ always admits a section in the 2-category above.
Fix an equivalence of $\VecOp$ with its skeleton, consisting of objects $\bm1^{\oplus n}$, $n\geq 0$.
On these objects, we simply define the section $s^*$ in the only way we can:
$\bm1$ has to go to $\bm1$,
and direct sums must be respected.
This is clearly a section,
since $f^*$ is monoidal and respects direct sums.
Moreover, we can explicitly describe the ind-right adjoint $s_*$,
which turns out to restrict to an honest right adjoint of $s^*$.
Indeed, there's a natural isomorphism $s_*\simeq\Hom_{\clg{A}}(\bm1,\blank)$.
This follows immediately from the fact that $\iHom_{\VecOp}(\bm1,\blank)$
is the identity functor on $\VecOp$.
\end{rmk}

Given a neutralized Tannakian category $(\clg{B},f^*)$,
Tannaka duality produces a pro-algebraic group,
which we denote $\scr{G}(f^*)$,
or $\scr{G}(\clg{B},f^*)$, or $\scr{G}(\clg{B})$, depending on what we want to emphasize.
We call this group the \emph{Tannaka dual} of $(\clg{B},f^*)$.
This group is such that $f^*$ naturally factors
through the forgetful functor $\Rep \scr{G}(f^*)\to\VecOp$,
and such that $\clg{B}\to\Rep \scr{G}(f^*)$ is a (monoidal) $F$-linear equivalence.
We also have a canonical isomorphism $\scr{G}(\Rep G)\simeq G$ for all $G$.

\begin{rmk}[On Tannakian categories and unipotency]\label{rmk:tannaka-unip}\ \\
\vspace{-1em}
\begin{itemize}
\item A pro-algebraic group $G$ is pro-unipotent if and only if
$\Rep G$ is unipotent.
\item All objects in a neutral Tannakian category are of finite length,
so for these categories, weak unipotency and unipotency coincide.
\item If $V$ is a unipotent $G$-representation, then any subquotient of $V$ in $\Rep G$
is unipotent. (Proof: Enough to check for subobjects of $V$, by rigidity.
Given a unipotent filtration on $V$,
the induced filtration on a subobject is also a unipotent filtration,
by the fact that any subobject of $\bm1^n$ is of the form $\bm1^m$.)\qed
\end{itemize}
\end{rmk}

Tannaka duality is functorial:
given a morphism $p^*$ of neutralized Tannakian categories as above,
we get an induced group-scheme homomorphism
$p\colon \scr{G}(\varphi^*)\to \scr{G}(f^*)$,
which we always write using the same symbol,
but without the asterisk.
Conversely,
given a group homomorphism $p\colon G\to H$,
we get a morphism of Tannakian categories $p^*\colon\Rep H\to\Rep G$,
namely restriction along $p$.
The ind-right adjoint $p_*$ is in this context given by induction along $p$
(cf.\ \nref{Section}{sec:res-ind-inv}).

The following proposition is an example of how properties of group homomorphisms
are reflected on the Tannakian side:
\begin{prop}[{\cite[Prop.~2.21]{deligne-milne}}]\label{prop:tannaka-mor-cor}
Let $p\colon G\to H$ be a homomorphism of pro-algebraic groups.
Then,
\begin{enumerate}
\item $p$ is a closed immersion (injective) if and only if
every object in $\Rep G$ is isomorphic to a subquotient of an object in the image of $p^*$,
and
\label{enum:injective}
\item $p$ is faithfully flat (surjective) if and only if
for every subobject
$X\hookrightarrow p^*Y$ in $\Rep G$,
there exists a subobject $Y'$ of $Y$ such that $p^*Y' \simeq X$.
\label{enum:surjective}
\end{enumerate}
\end{prop}
\begin{proof}
This is essentially shown in \cite[Prop.~2.21]{deligne-milne},
the only difference being that the second part also includes the condition that $p^*$ is full.
However,
fullness follows from having a subobject stable image.
Indeed,
let $Y, Y'$ be representations of $H$,
and let $\varphi\colon p^*Y \to p^*Y'$ be a morphism.
It corresponds to a subobject
$\Gamma_\varphi \subseteq p^*Y \oplus p^*Y' \simeq p^*(Y \oplus Y')$,
namely its graph.
Lifting this to a subobject of $Y \oplus Y'$
amounts exactly to
lifting $\varphi$ to a map $Y \to Y'$.
\end{proof}

With the second part of this proposition in mind,
we make the following definition:
\begin{defn}\label{defn:tannak_subcat} 
A \emph{Tannakian subcategory} of a Tannakian category is
a full abelian subcategory closed under direct sums, tensor products, taking duals, and taking subquotients.
\end{defn}

Given a homomorphism $G \to H$ of pro-algebraic groups,
we sometimes denote by
$\Res^H_G \colon \Rep H \to \Rep G$
the restriction functor from $H$-representations to $G$-representations.

\begin{lem}\label{lem:im_of_res_from_im}
Given a homomorphism $p\colon G\to H$,
the smallest Tannakian subcategory of $\Rep G$ containing $\im(p^*)$
is the essential image of $\Res^{\im(p)}_G$.
\end{lem}
\begin{proof}
Let $\clg{T}\subseteq\Rep(G)$ denote the smallest Tannakian subcategory containing $\im(p^*)$.
Note that we have a factorization
$G \to \scr{G}(\clg{T}) \to H$
of $p$,
and that the first map is surjective by \nref{Prop.}{prop:tannaka-mor-cor}.\ref{enum:surjective}.
Next,
we check that
$\scr{G}(\clg{T}) \to H$
is injective by \nref{Prop.}{prop:tannaka-mor-cor}.\ref{enum:injective}.
Thus,
we need to show that every object in $\clg{T}$
is a subquotient of an object in the essential image of $p^*$.
But this follows from the following elementary fact:
if you take a full, monoidal subcategory of a Tannakian category,
closed under taking duals,
then the closure of this subcategory under taking subquotients is a Tannakian category.
Using this fact for the monoidal subcategory $\im(p^*)$ of $\Rep G$,
we get that $\clg{T}$ is its closure under forming subquotients,
by the minimality of $\clg{T}$.
\end{proof}

\begin{lem}\label{lem:unip_ker}
A homomorphism $p\colon G\to H$ has pro-unipotent kernel
if and only if
the category $\Rep G$ is $\Res^{\im(p)}_G$-unipotent.
(Equivalently,
$\Rep G$ is generated under extensions and subquotients by its full subcategory $\im(p^*)$.)
\end{lem}
\begin{proof}
Assume that $\Rep G$ is $\Res^{\im(p)}_G$-unipotent,
and let $K := \ker(p)$.
Then it is clear that whatever is in the image of $\Res^G_K$,
is (absolutely) unipotent.
By \nref{Prop.}{prop:tannaka-mor-cor}.\ref{enum:injective},
every $K$-representation is then a subquotient of a unipotent $K$-representation.
But such a representation must also be unipotent (\nref{Rmk.}{rmk:tannaka-unip}).

Now assume that $\ker(p)=:K$ is pro-unipotent.
Then, given a non-zero $G$-representation $V$,
the $K$-invariants $V^K$ are non-zero and form a sub-$G$-representation
(since $K$ is normal in $G$).
Moreover, $V^K$ is actually the restriction of an
$\im(p)$-representation,
since $\im(p) = \coker(K \to G)$.
This shows $\Res^{\im(p)}_G$-unipotency of $\Rep G$,
by \nref{Rmk.}{rmk:tannaka-unip}.
\end{proof}

One perspective on Tannaka duality is as follows.
Let $f^*$ be a fibre functor on a neutral Tannakian category $\clg{B}$ as above.
Then, using the 2-section $s^*$ of $f^*$ described in \nref{Rmk.}{rmk:fibre-section},
one can prove that $f^*f_*\bm1$ is naturally a Hopf algebra,
and we denote it by $\scr{H}(f^*) := f^*f_*\bm1$.
The Tannaka dual group is then $\scr{G}(f^*):=\Spec\scr{H}(f^*)$.
Note that $\scr{G}(f^*)$-representations are the same as $\scr{H}(f^*)$-comodules.

\begin{rmk}\label{rmk:regrep}
It's worth noting at this point,
that if $f^*$ denotes the canonical fibre functor on $\Rep G$,
then $f^*f_*\bm1$ is $\clg{O}(G)$,
and $f_*\bm1$ is the regular representation of $G$.
Moreover, $f_*\bm1$ is naturally an algebra object in $\Ind\Rep G$.
\end{rmk}

\subsection{Restriction, induction, and invariants}\label{sec:res-ind-inv}
Let $p\colon H\to G$ be a morphism of pro-algebraic groups.
We denote the restriction functor $\Rep G\to \Rep H$ by $p^*$,
or sometimes $\Res^G_H$ (usually when the morphism is unlabelled).
In the literature,
induction is often only defined for $p$ injective.
We depart from this convention and
call the ind-right adjoint $p_*\colon\Ind\Rep H\to\Ind\Rep G$ of $p^*$ \emph{induction}
for arbitrary $p$.
For this reason,
we sometimes denote $p_*$ by $\Ind_H^G$
(again, usually when the morphism is unlabelled).
Note that we can decompose any $p$ as a surjection
$p_0\colon H\to \im(p)$
followed by an injection $p_1\colon\im(p)\to G$.
Our induction functor $p_*$ is then given
by the composition of
the functor $\Gamma(\ker(p_0),\blank) = (\blank)^{\ker(p_0)}$
taking invariants under the action of the kernel of $p_0$,
and the more traditional induction functor $\Ind_{\im(p)}^G$ along the injective homomorphism $p_1$.

\begin{rmk}
The invariants functor $\Gamma(G,\blank) = (\blank)^G$ on $\Rep G$ is given by ``induction'' along $G\to 1$,
and the trivial representation functor $\VecOp\to\Rep G$ is given by restriction along $G\to 1$.
\end{rmk}

An important fact is that the projection morphism for the restriction-induction adjunction is an isomorphism.

\begin{prop}[The restriction-induction projection formula]\label{prop:proj}
Let $p$ be a homomorphism of pro-algebraic groups.
Then there is an isomorphism $p_*(A\otimes p^*B) \simeq p_*A \otimes B$,
natural in $A$ and $B$.
In particular, setting $A=\bm1$ gives a natural isomorphism
$p_*p^*\simeq p_*\bm1\otimes (\blank)$.
\end{prop}
\begin{proof}
For $p$ injective, this is \cite[{}I.3.6]{jantzen}.
For $p$ surjective, it is obvious.
Combining these two cases, we can conclude for arbitrary $p$.
\end{proof}

We give two trivial consequences of this result, since they will be useful in the sequel.
They both concern morphisms $p$ that admit a retraction $r$,
and that are therefore, in particular, injective.
The first result gives a description of $p^*$ in terms of $p_*$ and $r_*$.
The second says that,
when tensoring with $p_*\bm1$,
one may first change the $\ker(r)$-action to be trivial,
and then tensor with $p_*\bm1$,
and get the same result.

\begin{lem}\label{lem:proj-conseq}
Let $p\colon H\to G$ be a homomorphism of pro-algebraic groups,
and let $r\colon G\to H$ be a retraction of $p$.
Then we have a natural isomorphism $p^* \simeq r_*(p_*\bm1 \otimes \blank)$.
\end{lem}
\begin{proof}
By \nref{Prop.}{prop:proj},
we have an isomorphism
$p_*p^* \simeq p_*\bm1 \otimes (\blank)$.
To this isomorphism, we apply $r_*$,
then use that $r_*p_*=\id$, and get
$p^*\simeq r_*(p_*\bm1 \otimes \blank)$.
\end{proof}

\begin{lem}\label{lem:tensor-with-reg-rep}
Let $p\colon H\to G$ be a homomorphism of pro-algebraic groups,
and let $r\colon G\to H$ be a retraction of $p$.
Then we have a natural isomorphism $p_*\bm1 \otimes (\blank) \simeq p_*\bm1\otimes r^*p^*(\blank)$.
\end{lem}
\begin{proof}
Using the projection formula,~\ref{prop:proj},
we compute
\begin{align*}
p_*\bm1\otimes (\blank)
  \simeq p_*p^*(\blank)
  \simeq p_*(p^*r^*)p^*(\blank)
  \simeq p_*\bm1\otimes r^*p^*(\blank).
\end{align*}
\end{proof}

\subsection{Group isomorphisms}\label{sec:group-iso}
Consider a homomorphism of pro-algebraic groups
$p\colon G\to H$.
We have a corresponding morphism of Tannakian categories
$p^*\colon\Rep H\to \Rep G$ (restriction along $p$).
Let $f^*$ and $\varphi^*$ be the canonical fibre functors of $\Rep H$ and $\Rep G$, respectively.
Then $p^*$ induces a morphism of algebra objects
$p_\rom{alg}\colon p^*f_*\bm1 \to \varphi_*\bm1$,
from the restriction along $p$ of the regular representation of $H$ to the regular representation of $G$
(cf.\ \nref{Rmk.}{rmk:regrep}).
This map is simply the algebra homomorphism $\clg{O}(H)\to\clg{O}(G)$ induced by $p$,
which respects the $\clg{O}(G)$-coactions,
making it a morphism of algebra objects in $\Ind\Rep G$.

Another way of viewing this morphism,
is via the universal property of $\varphi_*\bm1$.
We have a natural isomorphism $\Hom(\blank,\varphi_*\bm1)\simeq \Hom(\varphi^*(\blank),\bm1)$.
So, to give a map $p^*f_*\bm1\to \varphi_*\bm1$,
is the same as to give a map from $\varphi^*p^*f_*\bm1$ to $\bm1$.
But $\varphi^*p^*f_*\bm1 \simeq f^*f_*\bm1$ has a natural map to $\bm1$:
the counit it's equipped with as a Hopf algebra.
In this way, we recover the algebra homomorphism
$p_\rom{alg}\colon p^*f_*\bm1 \to \varphi_*\bm1$
above.
We can see this by applying $\varphi^*$ to $p_\rom{alg}$,
and noting that the result is a morphism of Hopf algebras,
and is hence compatible with the two counits.
We get the following proposition:
\begin{prop}\label{prop:group-iso}
A homomorphism $p\colon G\to H$ of pro-algebraic groups is an isomorphism
if and only if the algebra morphism $p^*\clg{O}(H) \to \clg{O}(G)$ in $\Ind\Rep G$,
induced by the counit of $\clg{O}(H)$ under the universal property of $\clg{O}(G)$,
is an isomorphism.
\end{prop}

\subsection{Categories of modules}\label{sec:modules}
Consider an abelian tensor category $\clg{A}$ with all filtered colimits,
in which the tensor product commutes with finite colimits,
and let $\clg{A}^\rom{cpt}$ denote the full subcategory of compact objects:
objects $C\in \clg{A}$ such that $\Hom(C, \blank)$ commutes with filtered colimits.
Take an algebra object $A$ in $\clg{A}$,
and consider the category $\Mod(A)$ of (left) $A$-modules in $\clg{A}$,
i.e., objects $M\in\clg{A}$ together with an $A$-action $A\otimes M\to M$ satisfying the usual axioms.
The category $\Mod(A)$ is an abelian tensor category (under $\otimes_A$).
\begin{defn}
Let $M$ be an $A$-module, as above. Then,
\begin{itemize}
\item $M$ is \emph{strictly free of finite rank} if $M\simeq A^{\oplus n}$ for some $n$;
\item $M$ is \emph{free} if $M\simeq A\otimes X$ for some $X$ in $\clg{A}$;
\item $M$ is \emph{finitely generated free} if $M\simeq A\otimes X$ for some $X$ in $\clg{A}^\rom{cpt}$;
\item $M$ is \emph{finitely generated} if $M$ is a quotient $A\otimes X\twoheadrightarrow M$
  of some finitely generated free module.
\end{itemize}
\end{defn}

\begin{rmk}
If $\clg{B}$ is an abelian category,
then $(\Ind\clg{B})^\rom{cpt}$ is exactly the essential image of $\clg{B}$ in $\Ind\clg{B}$.
This is the situation in which we will use these notions,
i.e., $\clg{B}$ will be an abelian tensor category, $\clg{A}$ will be $\Ind\clg{B}$,
and $A$ will be an algebra object in $\Ind\clg{B}$.
\end{rmk}

\subsection{The Standard Situation}\label{sec:standard}
We now describe the Tannaka dual of the group theoretic setup from
Theorems~\ref{thm:M_intro} and~\ref{thm:U_intro} from the introduction.
Consider the following diagram in $\nTan$:%
\footnote{%
  If the choice of symbols looks strange,
  it is because it has been made with certain applications in mind.
  The letters $\clg{M}$ and $\clg{T}$ were chosen with the words \emph{motivic} and \emph{topological} in mind,
  while $B$ and $b$ were taken from \emph{Betti}.
  The letters $f$ and $\varphi$ are meant to evoke \emph{fibre},
  while $s$ and $\sigma$ simply stand for \emph{section}.%
}
\begin{equation}\begin{tikzcd}\label{diag:standard}
\clg{M}\ar[d,"f^*"'] \ar[r,"B^*"] &[1em] \clg{T} \ar[d,"\varphi^*"'] \\[1em]
\clg{E}\ar[u,"s^*"',bend right] \ar[r,"b^*"'] & \VecOp \ar[u,"\sigma^*"',bend right]
\end{tikzcd}\end{equation}
We sometimes refer to
$B^*$ and $b^*$ as \emph{realization functors}.
We assume we're given monoidal isomorphisms
$b^*f^* \simeq \varphi^*B^*$,
$\varphi^*\sigma^* \simeq \id$,
and $f^*s^* \simeq \id$.
These combine to give various other isomorphisms, such as $b^*\simeq \varphi^*B^*s^*$.
Moreover, $\sigma^*$ and $s^*$ are automatically full,
since they are 2-sections of faithful functors.
We denote the Tannaka duals as follows:
\begin{align*}
H:=\scr{G}(\clg{E}),\quad G:=\scr{G}(\clg{M}),\quad P:=\scr{G}(\clg{T}).
\end{align*}

Let us extract from diagram \heqref{diag:standard} the sequence
\begin{align*}
\clg{E} \xrightarrow{s^*} \clg{M} \xrightarrow{B^*} \clg{T},
\end{align*}
of morphisms of neutral Tannakian categories.
We have a dual sequence of groups
\begin{equation*}\begin{tikzcd}
P\ar[r,"B"] & G\ar[r,"s"] & H.
\end{tikzcd}\end{equation*}
Note that $s$ is surjective, since the 2-retraction $f^*$ of $s^*$ in $\nTan$
produces a section $f$ of $s$.%
\footnote{%
  It is somewhat unfortunate that $f$ denotes the section of $s$, instead of the other way around.
  But since the bulk of our arguments takes place on the Tannakian side of things,
  it seemed reasonable to give priority in the notation to the fact that $s^*$ is a section of $f^*$.%
}
Writing $\iota\colon K\to G$ for the kernel of $s$,
we get a split exact sequence of pro-algebraic groups
\begin{equation*}\begin{tikzcd}
1\ar[r] &[-1em] K\ar[r,"\iota"] & G\ar[r,"s"'] & H\ar[l,"f"',bend right] \ar[r] &[-1em] 1,
\end{tikzcd}\end{equation*}
and an isomorphism $G\simeq K \rtimes H$.

The splitting just described induces an action by $H$ on $K$,
telling us that $\clg{O}(K)$ should naturally come from $\Ind\clg{E}$.
Indeed, the fact that $f^*$ has a nice section $s^*$,
gives rise to a natural Hopf algebra structure on $f^*f_*\bm1$ in $\Ind\clg{E}$
(see~\cite[Thm.~1.45]{motivic_hopf_I}),
and this Hopf algebra then describes the kernel group $K$.
More precisely, $\Spec(b^*(f^*f_*\bm1))\simeq K$
(see~\cite[Thm.~2.7]{motivic_hopf_II}).
In analogy with the notation introduced in \nref{Section}{sec:tannaka},
we write $f^*f_*\bm1 =: \scr{H}(\clg{M}/\clg{E},f^*,s^*)$,
omitting some of the arguments from the notation depending on what we want to emphasize.
Likewise, we write $\scr{G}(\clg{M}/\clg{E},f^*,s^*)$ for its pro-dual in $(\Ind\clg{E})^\rom{op} \simeq \Pro(\clg{E}^\rom{op})$.
Then one can make sense of the formula $b^*\scr{G}(\clg{M}/\clg{E})\simeq K$.

Note that the composition $B^*\circ s^*$ factors through $\sigma^*$,
up to a natural isomorphism.
Thus, the composition $s\circ B$ is trivial,
and $B$ factors as $B=\iota\circ\varrho$, through the kernel $K$ of $s$.
The following diagrams summarize the situation:
\begin{equation*}\begin{tikzcd}
1 \ar[r]
  &[-1em]
    K \ar[r, "\iota"]
  & G \ar[r, "s"']
  & H \ar[l, bend right, "f"'] \ar[r]
  &[-1em] 1 \\
&& P\ar[ul, dotted, bend left=24, "\exists\varrho"'] \ar[u, "B" ']
\end{tikzcd}
\qquad\quad
\begin{tikzcd}
\clg{M} \ar[d, "f^*" '] \ar[r, "\iota^*"]
  & \Rep K \ar[d, "g^*"'] \ar[r, "\varrho^*"]
  & \clg{T} \ar[d, "\varphi^*" '] \\
\clg{E} \ar[u, bend right, "s^*"'] \ar[r, "b^*"'] & \VecOp \ar[u, bend right, "t^*"'] \ar[r, equals] & \VecOp \ar[u, bend right, "\sigma^*"']
\end{tikzcd}
\end{equation*}

Finally,
there is a canonical map
\begin{align}\label{eq:canonical_map}
b^* Rs_* \to R\sigma_* B^*,
\end{align}
which will be important to us.
It corresponds under adjunction to the map
$\sigma^*b^*Rs_* \to B^*$
which,
via the isomorphism $\sigma^*b^* \simeq B^*s^*$,
is induced by the counit of the $(s^*, Rs_*)$-adjunction.
In our two main theorems,
the key assumption is that the degree 0 and 1 parts of this canonical map~\heqref{eq:canonical_map}
are isomorphisms.

\section{The unipotent isomorphism criterion}\label{sec:thmU}
The goal of this section is to prove Thm.~\ref{thm:U_intro}.
The methods are Tannakian,
and the first step is to reformulate the result under Tannakian duality,
via the Standard Situation (\nref{Section}{sec:standard}):
\begin{equation*}\begin{tikzcd}
\clg{M} \ar[r, "B^*"] \ar[d, "f^*" '] &[1em] \clg{T} \ar[d, "\varphi^*" '] \\[1em]
\clg{E} \ar[r, "b^*"'] \ar[u, bend right, "s^*"'] & \VecOp \ar[u, bend right, "\sigma^*"']
\end{tikzcd}\end{equation*}
We may now state the first key result.
\begin{thm}\label{thm:U1}
Suppose $\clg{T}$ is unipotent.
If the canonical maps $b^*R^is_* \to R^i\sigma_*B^*$, for $i=0,1$, are isomorphisms,
then the natural algebra homomorphism
$B^*f_*\bm1 \to \varphi_*\bm1$
(see \nref{Section}{sec:group-iso})
is an isomorphism.
\end{thm}

\begin{rmk}\label{rmk:automatic-unipotency}
The assumptions of \nref{Thm.}{thm:U1} imply that $\clg{M}$ is $s^*$-unipotent,
as follows.
By \nref{Rmk.}{rmk:tannaka-unip},
it's enough to show weak unipotency.
But given a non-zero object $X$ of $\clg{M}$,
the subobject $s^*s_*X$ must be non-zero.
(The fact that it's a subobject follows from the surjectivity of
$\scr{G}(\clg{M}) \to \scr{G}(\clg{E})$.)
To check this, we apply $B^*$,
yielding $\sigma^*\sigma_*B^*X$,
Now the first assumption give that this is non-zero,
since $B^*X$ is non-zero and $\sigma^*\sigma_*B^*X$ is the largest subobject of $B^*X$
in the essential image of $\sigma^*$.
\end{rmk}

Recall, that in the Standard Situation (\nref{Section}{sec:standard}),
there is an induced group homomorphism
\begin{align*}
\varrho\colon P=\scr{G}(\varphi^*) \to b^*\scr{G}(f^*) = K.
\end{align*}
Using \nref{Thm.}{thm:U1} and \nref{Prop.}{prop:group-iso},
we get \nref{Thm.}{thm:U_intro} from the introduction,
which in the present context reads as follows.
\begin{thm}[Unipotent isomorphism criterion]\label{thm:U2}
Suppose $P$ is pro-unipotent.
If the canonical maps $b^*R^is_* \to R^i\sigma_*B^*$, for $i=0,1$, are isomorphisms,
then $\varrho$ is an isomorphism.
\qed
\end{thm}

We prove \nref{Thm.}{thm:U1} by constructing $U\simeq \varphi_*\bm1$ in $\Ind\clg{T}$
(\nref{Section}{subsec:U} and \nref{Prop.}{prop:characterizationU})
and $W\simeq f_*\bm1$ in $\Ind\clg{M}$
(\nref{Section}{subsec:W} and \nref{Prop.}{prop:characterizationW})
as well as an isomorphism $B^*W\simeq U$
(\nref{Prop.}{prop:realization}).
After all of that,
we only have to check some compatibilities,
which we do in \nref{Section}{sec:unipotent-proof}.

\subsection{Construction of $U$}\label{subsec:U}
The relevant context in this section is that of a Tannakian category
$\clg{T}$ with fibre functor $\varphi^*$,
which has a section $\sigma^*$ as in \nref{Rmk.}{rmk:fibre-section}.
We need no further assumptions for the construction.

We begin with some motivation.
When $\clg{T}$ is assumed to be unipotent,
every object is (by definition) a successive extension of trivial objects, starting from $\bm1$.
The ind-object $\varphi_*\bm1$ is also (ind-)unipotent in this situation.
It moreover (always) satisfies a universal property that makes it as ``large'',
in the sense that every object in $\clg{T}$ appears
as a subquotient of some $(\varphi_*\bm1)^{\oplus n}$.
The idea, when constructing $U$,
is therefore to start from $\bm1$
and iterate a kind of universal extension-of-a-trivial-object operation,
to get a large, universal, ind-unipotent object.

Let $U_0 := \bm{1}_{\clg{T}}$ and construct $U_{n+1}$ inductively:
\begin{align}\label{eq:undefseq}
0\to U_n\to U_{n+1}\to \sigma^*V\to 0,
\end{align}
i.e., define $U_{n+1}$ as an extension of a trivial object by $U_n$.
We have to say what $V$ is and what the extension class is.

Note that
\begin{align*}
\Ext^1(\sigma^*V, U_n) &= \Hom(\sigma^*V,U_n[1]) \\
&= \Hom(V,R\sigma_* U_n[1]) \\
&= \Hom(V,R^1\sigma_* U_n),
\end{align*}
where the last equality is due to $\VecOp$ being semisimple.
We therefore take $V$ to be $R^1\sigma_*U_n$
and the extension class to be given by the identity in $\End(R^1\sigma_*U_n)$.
Set $U := \colim_{n\geq 0} U_n$ in $\Ind\clg{T}$.

\begin{rmk}
These are \emph{universal extensions} in the sense of~\cite[Def.~3.1]{daddezio-esnault},
but we won't use this fact.
\end{rmk}

We now record two key properties of this construction.
Apply $\sigma_*\simeq\Hom(\bm 1, \blank)$ to the defining sequence~\heqref{eq:undefseq} of $U_{n+1}$ to get a long exact sequence
\begin{equation*}\begin{tikzcd}
0 \ar[r] & \Hom(\bm 1, U_n) \ar[r, "a"] & \Hom(\bm 1, U_{n+1}) \ar[r] & R^1\sigma_*U_n \\
\ar[r, "\delta"] & \Ext^1(\bm 1, U_n) \ar[r , "b"] & \Ext^1(\bm 1, U_{n+1}) \ar[r] & \dotsm.
\end{tikzcd}\end{equation*}
We have that $\delta$ is an isomorphism, by definition.
Therefore, $a$ is an isomorphism and $b$ is the zero map.
This establishes the following lemma:
\begin{lem}\label{lem:triv-pushout-U}
We have a canonical isomorphism $\bm1\simeq\sigma_*U_0$,
and the canonical maps $\sigma_*U_n\to\sigma_*U_{n+1}$, $n\geq 0$, are isomorphisms.
Moreover, the pushout extension along $U_n\to U_{n+1}$ of any extension of a trivial object by $U_n$
is split.
\qed
\end{lem}

\subsection{Construction of $W$}\label{subsec:W}
For this construction,
we work with Tannakian categories $\clg{M}$ and $\clg{E}$
and a morphism $f^*\colon \clg{M} \to \clg{E}$ with a section $s^*$,
as introduced above.
We need no further assumptions for the construction.

The idea for this construction is simply to implement the construction of $U$
in this more general context.
The difficulty arises from the failure of $\clg{E}$ to be semisimple
(in contrast to $\VecOp$).
While not every distinguished triangle in the derived category $\clg{D}(\clg{E})$ of $\clg{E}$ splits,
the ones we need to consider, it turns out, do split.
The functor $f^*$ plays a key role in finding these splittings.

Let $W_0:=\bm 1_{\clg{M}}$.
We define $W_{n+1}$ inductively from $W_n$ as an extension
\begin{align}\label{eq:wndefseq}
0\to W_n\to W_{n+1}\to s^*R^1s_*W_n\to 0.
\end{align}
Such an extension is given by a map
\begin{align*}
s^*R^1s_*W_n \to W_n[1],
\end{align*}
in the derived category.
The counit of the $(s^*,s_*)$-adjunction gives us a map
$s^*Rs_*W_n[1] \to W_n[1]$,
and what's missing is a map
$s^*R^1s_*W_n \to s^*Rs_*W_n[1]$.
This map will be the image under $s^*$ of a map
$R^1s_*W_n \to Rs_*W_n[1]$
which we now spend some time constructing.

We shall make two auxiliary induction hypotheses.
\begin{enumerate}[({IH-}1)]
\item\label{enum:ih1}
We have a canonical isomorphism $\bm1\simeq s_*W_0$,
and the canonical maps
\begin{align*}
s_*W_0\to s_*W_1 \to \dotsm \to s_*W_n,
\end{align*}
are isomorphisms.
\item\label{enum:ih2}
We have a map $\alpha_n\colon f^*W_n\to\bm 1$,
such that the composition
\begin{align*}
\bm 1 \simeq s_*W_n \simeq f^*s^*s_* W_n \to f^*W_n \xrightarrow{\alpha_n} \bm 1,
\end{align*}
is the identity map.
\end{enumerate}

These hypotheses are satisfied for the base case $n=0$:
$s_*\bm1\simeq s_*s^*\bm1 \simeq \bm1$,
since $s^*$ is monoidal and fully faithful;
and we let $\alpha_0\colon f^*\bm1\simeq \bm1$
be the map coming from the monoidal structure of $f^*$.

We are now ready to construct the missing map
$R^1s_*W_n \to Rs_*W_n[1]$,
using the auxiliary assumptions on $W_n$.
They allow us to split off $s_*W_n$ from $Rs_*W_n$.

Note that
\hyperref[enum:ih1]{(IH-\ref*{enum:ih1})}
gives us a map $\bm 1\to Rs_*W_n$,
which is an isomorphism on the zeroth cohomology.
Now consider a cone $C$ of this map:
\begin{align}\label{eq:ctriangle}
\bm 1 \to Rs_*W_n \to C \to.
\end{align}
Moreover, consider the map $\beta_n\colon Rs_*W_n\to f^*W_n$ defined as the composition
\begin{align*}
Rs_*W_n\simeq f^*s^*Rs_*W_n\to f^*W_n.
\end{align*}
By
\hyperref[enum:ih2]{(IH-\ref*{enum:ih2})},
the triangle~\heqref{eq:ctriangle} is split by retracting the left map,
$\bm1\to Rs_*W_n$,
with
\begin{align*}
Rs_* W_n \xrightarrow{\beta_n} f^*W_n \xrightarrow{\alpha_n} \bm 1,
\end{align*}
and we hence get a section $s\colon C\to Rs_*W_n$.
By construction, the first non-zero cohomology of $C$ is $H^1C = R^1s_*W_n$
and we get a map $\iota\colon R^1s_*W_n\to C[1]$.
Composing with $s[1]\colon C[1]\to Rs_*W_n[1]$ gives the map
$R^1s_*W_n\to Rs_*W_n[1]$ we are looking for.
The following diagram may clarify the situation:
\begin{equation*}\begin{tikzcd}
\bm 1 \ar[rr] && Rs_*W_n \ar[dl, bend left, "\beta_n"] \ar[r] & C \ar[l, bend right, dotted, "s" '] \\
& f^*W_n \ar[ul, bend left, "\alpha_n"] && R^1s_*W_n[-1] \ar[u, "\iota{[-1]}" ']
\end{tikzcd}\end{equation*}

What is left is to establish the induction hypotheses
\hyperref[enum:ih1]{(IH-\ref*{enum:ih1})} and \hyperref[enum:ih2]{(IH-\ref*{enum:ih2})}
for $W_{n+1}$.
We start with lifting $\alpha_n$ to $W_{n+1}$.
Applying $f^*$ and then the contravariant $\Hom(\blank,\bm 1)$
to the defining sequence~\heqref{eq:wndefseq} of $W_{n+1}$
yields an exact sequence
\begin{equation*}\begin{tikzcd}[row sep=0]
\Hom(f^*W_{n+1},\bm 1) \ar[r] & \Hom(f^*W_n,\bm 1) \ar[r, "a"] & \Ext^1(R^1s_*W_n,\bm 1) \\
\alpha_{n+1} \ar[r, mapsto, dotted] & \alpha_n \ar[r, mapsto] & a(\alpha_n)
\end{tikzcd}\end{equation*}
and hence the existence of $\alpha_{n+1}$ is equivalent to $a(\alpha_n)$ being zero.
But $a(\alpha_n)=(\alpha_n[1])\circ\omega_n$,
where $\omega_n$ is the extension class of $f^*W_{n+1}$, i.e.,
\begin{align*}
\omega_n\colon R^1s_*W_n \xrightarrow{\iota} C[1]
  \xrightarrow{s[1]} Rs_*W_n[1]
  \xrightarrow{\beta_n[1]} f^*W_n[1].
\end{align*}
Composing this with $\alpha_n[1]$ gives zero,
because $\alpha_n\circ\beta_n\circ s = 0$,
as this was how we split the triangle~\heqref{eq:ctriangle} defining $C$.
Finally, consider the following diagram:
\begin{equation*}\begin{tikzcd}
\bm 1 \ar[d, equals] \ar[r, "\sim"] & s_*W_n \ar[d] \ar[r, "\sim"] & f^*s^*s_* W_n \ar[d] \ar[r] & f^*W_n \ar[d] \ar[r, "\alpha_n"] & \bm 1 \ar[d, equals] \\
\bm 1 \ar[r] & s_*W_{n+1} \ar[r] & f^*s^*s_* W_{n+1} \ar[r] & f^*W_{n+1} \ar[r, "\alpha_{n+1}"] & \bm 1
\end{tikzcd}\end{equation*}
It is commutative,
and hence the bottom row is the identity as required,
establishing
\hyperref[enum:ih2]{(IH-\ref*{enum:ih2})}
for $W_{n+1}$.

Lastly, as the full faithfulness of $s^*$ implies $s_*s^*\simeq\id$,
the long exact sequence after applying $s_*$ to~\heqref{eq:wndefseq} is:
\begin{equation*}\begin{tikzcd}
0 \ar[r] & s_* W_n \ar[r, "a"] & s_*W_{n+1} \ar[r] & R^1s_*W_n \\
\ar[r, "\delta"] & R^1s_* W_n \ar[r , "b"] & R^1s_*W_{n+1} \ar[r] & \dotsm.
\end{tikzcd}\end{equation*}
Because $\delta$ is the identity map by construction,
we get that $a\colon s_*W_n\simeq s_*W_{n+1}$ is an isomorphism,
establishing
\hyperref[enum:ih1]{(IH-\ref*{enum:ih1})}
for $W_{n+1}$.
This finishes the construction of $W := \colim_{n\geq 0} W_n$.

\subsection{Properties}
In this section, we establish some key properties of $U$ and $W$.
Most importantly,
their isomorphisms with $\varphi_*\bm1$ and $f_*\bm1$,
respectively,
are proved in \nref{Prop.}{prop:characterizationU} and \nref{Prop.}{prop:characterizationW}.

\subsubsection{$W$ realizes to $U$}
We begin with establishing that $B^*W$ is
isomorphic to $U$.
\begin{prop}\label{prop:realization}
If the canonical map $b^*R^1s_* \to R^1\sigma_*B^*$ is an isomorphism,
then $B^*W\simeq U$.
\end{prop}
\begin{proof}
We give an isomorphism of inductive systems $(B^*W_n)_n\simeq (U_n)_n$.
Clearly $B^*W_0\simeq U_0$, since $B^*$ is monoidal.
We now proceed by induction, and assume that $B^*W_n\simeq U_n$.
Applying $B^*$ to the defining sequence~\heqref{eq:wndefseq} for $W_{n+1}$
yields an isomorphism of extensions
\begin{equation*}\begin{tikzcd}
0 \ar[r] & B^*W_n \ar[d, "\sim" {rotate=90, anchor=north}] \ar[r]
         & B^*W_{n+1} \ar[d, "\sim" {rotate=90, anchor=north}] \ar[r]
         & B^*(s^*R^1s_*W_n) \ar[d, "\sim" {rotate=90, anchor=north}] \ar[r] & 0 \\
0 \ar[r] & U_n \ar[r] & E \ar[r] & \sigma^*R^1\sigma_*U_n \ar[r] & 0
\end{tikzcd}\end{equation*}
where we've used that realization commutes with pushforwards on the rightmost term.
We have to check that the extension class is correct.
Applying $B^*$ to the composition
\begin{align*}
s^*R^1s_*W_n \to s^*C[1] \to s^*Rs_*W_n[1] \to W_n[1],
\end{align*}
gives
\begin{align*}
\sigma^*R^1\sigma_* U_n \to \sigma^*b^*C[1] \to \sigma^*b^*Rs_*W_n[1] \to U_n[1].
\end{align*}
Under adjunction this corresponds to 
\begin{align*}
R^1\sigma_* U_n \to b^*C[1] \to b^*Rs_*W_n[1] \to R\sigma_*U_n[1].
\end{align*}
This map is in the semisimple category $\clg{D}(\VecOp)$,
and the source is concentrated in degree zero,
so it is equivalent to its zeroth cohomology
\begin{align*}
R^1\sigma_* U_n \to H^1(b^*C) \to b^*R^1s_*W_n \to R^1\sigma_*U_n,
\end{align*}
and, looking at the construction of $U$,
we have to see that this composition is the identity.
This is easy to see,
using the definition of $C$
together with the definition of the first two maps,
as well as the assumption of the proposition (to deal with the last map).
\end{proof}

\subsubsection{Characterization of $U$}
The universal property of $\varphi_*\bm1$ is that the natural transformation
\begin{align*}
\Hom(\blank,\varphi_*\bm1) \to \Hom(\varphi^*(\blank),\bm1),
\end{align*}
is an isomorphism.
This transformation takes a map on the left hand side,
applies $\varphi^*$ to it,
and then composes with $\varphi^*\varphi_*\bm1\to\bm1$.
In order to show the isomorphism $U\simeq \varphi_*\bm1$,
we simply have to establish the same universal property for $U$.

A key property of $\varphi_*\bm1$ (which we won't use) is the following:
\begin{prop}\label{prop:split-reg-rep}
Every short exact sequence
\begin{align*}
0\to \varphi_*\bm1 \to E \to X \to 0,
\end{align*}
in $\Ind\clg{T}$ splits.
\end{prop}
\begin{proof}
Since $\varphi^*$ is exact,
$\varphi_*\bm1$ is injective.
\end{proof}

\nref{Lemma}{lem:triv-pushout-U} implies that $U$ has a weaker variant of this property,
given by the following proposition:
\begin{prop}[Splitting Property]\label{prop:split-U}
Every short exact sequence of the form
\begin{align*}
0\to U\to E\to \sigma^*V\to 0,
\end{align*}
in $\Ind\clg{T}$, where $V$ is a finite-dimensional vector space,
splits.
\end{prop}
\begin{proof}
Without loss of generality,
we may replace $\sigma^*V$ by $\bm1$.
We have that $E$ is the filtered colimit of its subobjects lying in $\clg{T}$.
Let $L$ be a subobject of $E$, in $\clg{T}$,
which is not contained in (the image of) $U$
(it is clearly possible to find such an $L$).
Then, let $L' := L\cap U$,
and denote the cokernel of $L' \to L$ by $Q$.
We obtain an inclusion of short exact sequences,
\begin{equation*}\begin{tikzcd}
0 \ar[r] & L' \ar[d, hook] \ar[r] & L \ar[d, hook] \ar[r] & Q \ar[d, hook] \ar[r] & 0 \\
0 \ar[r] & U \ar[r]               & E \ar[r]              & \bm1 \ar[r]           & 0
\end{tikzcd}\end{equation*}
Thus $Q$ has to be either $0$ or $\bm1$,
and it can't be $0$ since $L\neq L'$ by assumption.
We therefore have our sequence as a pushout
along $L' \to U$
of this subsequence lying in $\clg{T}$.
The map $L' \to U$ factors through $U_n \to U_{n+1}$ for some $n$,
and we conclude by \nref{Lemma}{lem:triv-pushout-U}.
\end{proof}

Consider the map $u \colon \varphi^*U\to \bm1$,
obtained from $b^*\alpha_n$:
\begin{equation*}\begin{tikzcd}
b^*f^*W_n \ar[d, "b^*\alpha_n"] \ar[r, "\sim" inner sep=.3mm] & \varphi^*U_n \ar[dl, dotted, "u_n", bend left, near start] \\
\bm1
\end{tikzcd}\end{equation*}
By definition, $u_0\colon\varphi^*\bm1=\varphi^*U_0\to\bm1$ is the map coming from the monoidal structure of $\varphi^*$.
Using the splitting property (\nref{Prop.}{prop:split-U})
and the fact that $u_0\colon\varphi^*U_0\to \bm1$ is an isomorphism,
we can now prove that $U\simeq\varphi_*\bm1$.

\begin{prop}\label{prop:characterizationU}
Assume that $\clg{T}$ is unipotent.
Then, the natural transformation
\begin{align*}
\nu\colon \Hom(\blank, U) &\to \Hom(\varphi^*(\blank), \bm 1) \\
\psi &\mapsto u\circ\varphi^*(\psi)
\end{align*}
is an isomorphism.
In particular, the maps $U\to \varphi_*\bm1$ and $\varphi_*\bm1 \to U$
induced by $u\colon \varphi^*U\to\bm1$ and $\varphi^*\varphi_*\bm1\to\bm1$, respectively,
are mutually inverse isomorphisms.
\end{prop}
\begin{proof}
We need to prove that $\nu_L$ is bijective for every $L\in \clg{T}$.
The method of proof is induction on the unipotency index $n$ of $L$.
More precisely,
in the induction step
we assume that for an $L'$ with unipotency index strictly lower than $n$,
the map $\Hom(L',U)\to \Hom(\varphi^*L',\bm1)$,
sending $\psi'$ to $u\circ\varphi^*(\psi')$,
is a bijection.

\textbf{Surjectivity:}
\emph{Base case:} $L$ is isomorphic to some $\sigma^*V$.
We suppose we're given a map $\bar\psi\colon V\simeq\varphi^*\sigma^*V\to \bm1$.
Applying $\sigma^*$, we get a map $\psi\colon \sigma^*V\to \sigma^*\bm1=U_0\to U$.
Then $\nu(\psi) = u\circ\varphi^*(\psi) = \bar\psi$.

\emph{Induction step:}
Now suppose that $L$ has unipotency index $n>0$ and that
\begin{align}\label{eq:L-induction}
0\to L' \xrightarrow{a} L \xrightarrow{b} \sigma^*V \to 0
\end{align}
where $L'$ has a unipotency index of at most $n-1$.
We suppose we're given a map $\bar\psi\colon \varphi^*L\to \bm1$,
which we need to lift to a map $\psi\colon L\to U$ such that $\nu(\psi)=\bar\psi$.

Composing $\varphi^*(a)$ with $\bar\psi$ yields $\bar\psi'\colon \varphi^*L'\to \bm 1$.
By the induction hypothesis, $\bar\psi'$ can be lifted to $\psi'\colon L'\to U$,
such that $\nu(\psi')=\bar\psi'$.
We push out the short exact sequence~\heqref{eq:L-induction} via $\psi'$:
\begin{equation*}\begin{tikzcd}
0 \ar[r] & L' \ar[d, "\psi'"] \ar[r, "a"] & L \ar[d] \ar[r, "b"] & \sigma^*V \ar[d, equals] \ar[r] & 0 \\
0 \ar[r] & U \ar[r] & E \ar[l, bend left, dotted] \ar[r] & \sigma^*V \ar[r] & 0
\end{tikzcd}\end{equation*}
By \nref{Prop.}{prop:split-U}, the lower sequence must split: choose a splitting.
Denote the composed map $L\to U$ by $\psi_0$.
This is not quite the lift we're looking for (it depends on the choice of splitting).
We have $\psi_0 a=\psi'$, so that
\begin{align*}
u\circ\varphi^*(\psi_0)\circ\varphi^*(a)
  = u\circ\varphi^*(\psi')
  = \bar\psi'
  = \bar\psi\circ\varphi^*(a)
\end{align*}
This implies that $u\circ\varphi^*(\psi_0)-\bar\psi$ factors through $\varphi^*(b)$,
i.e., that there exists $\bar\psi''\colon V\to\bm 1$ such that
\begin{align*}
u\circ\varphi^*(\psi_0)-\bar\psi = \bar\psi''\circ\varphi^*(b).
\end{align*}
By induction, $\bar\psi''$ lifts to $\psi''\colon \sigma^*V\to U$,
and we define $\psi := \psi_0 - \psi''\circ b$.
This is the lift of $\bar\psi$ that we are looking for:
\begin{align*}
\nu(\psi)
  = u\circ\varphi^*(\psi)
  = u\circ\varphi^*(\psi_0-\psi''\circ b)
  = u\circ\varphi^*(\psi_0) - \bar\psi''\circ\varphi^*(b)
  = \bar\psi.
\end{align*}

\textbf{Injectivity:}
\emph{Base case:} $\varphi^*$ is fully faithful on the essential image of $\sigma^*$,
and $u_0$ is an isomorphism.

\emph{Induction step:}
Let $\psi\colon L\to U$ be such that $\nu(\psi):=u\circ\varphi^*(\psi)=0$.
Then, with $L'$ as before,
the composition $L'\to L\to U$ is zero, by induction.
Therefore, $\psi=\psi''\circ b$.
Since $b$ is epic and $\nu(\psi)=0$,
we get $\nu(\psi'')=0$,
and by induction $\psi''=0$,
implying $\psi=0$.
\end{proof}

\begin{rmk}
The above proof is inspired by the proof of
\cite[Prop.~1.17]{lazda},
which Lazda writes he copied from
\cite[Prop.~2.1.6]{hadian}.
\end{rmk}

We have the following useful property of $\varphi_*\bm1$,
which is also enjoyed by $U$, given the previous result:
\begin{lem}\label{lem:keyprop}
We have that
$
\sigma_*(\varphi_*\bm1 \otimes \blank) \simeq \varphi^*(\blank).
$
\end{lem}
\begin{proof}
Apply \nref{Lemma}{lem:proj-conseq}.
\end{proof}

\subsubsection{Characterization of $W$}
In this section, we prove that $W\simeq f_*\bm1$
under the assumptions of \nref{Thm.}{thm:U1}.
We are able to do this just using
the fact that $W$ realizes to $U$ (\nref{Prop.}{prop:realization}),
and the characterization of $U$ (\nref{Prop.}{prop:characterizationU}).

\begin{lem}\label{lem:keypropW}
We have that
$
s_*(f_*\bm1 \otimes \blank) \simeq f^*(\blank).
$
\begin{proof}
Apply \nref{Lemma}{lem:proj-conseq}.
\end{proof}
\end{lem}

By the universal property of $f_*\bm1$,
the map $\alpha\colon f^*W\to \bm1$
induces a map $W\to f_*\bm1$.
This then yields a natural transformation
\begin{align}\label{eq:inducedbyalpha}
s_*(W\otimes \blank) \to s_*(f_*\bm1 \otimes \blank).
\end{align}

\begin{lem}\label{lem:keylemmaforW}
Assume that $\clg{T}$ is unipotent and that
the canonical maps $b^*R^is_*\to R^i\sigma_*B^*$, for $i\leq 1$, are isomorphisms.
Then the natural transformation~\heqref{eq:inducedbyalpha} is an isomorphism.
\end{lem}
\begin{proof}
After applying $b^*$ we may view the transformation as the composition
\begin{align*}
b^*s_*(W\otimes \blank)
  \simeq \sigma_*(U\otimes B^*(\blank))
  \simeq \sigma_*(\varphi_*\bm1\otimes B^*(\blank))
  \simeq \varphi^*B^*
  &\simeq b^*f^* \\
  &\simeq b^*s_*(f_*\bm1 \otimes \blank),
\end{align*}
using
$b^*s_* \simeq \sigma_*B^*$,
\nref{Prop.}{prop:realization},
\nref{Prop.}{prop:characterizationU},
\nref{Lemma}{lem:keyprop},
and \nref{Lemma}{lem:keypropW}.
We conclude by the conservativity of $b^*$.
\end{proof}

\begin{prop}\label{prop:characterizationW}
Assume that $\clg{T}$ is unipotent and that
the canonical maps $b^*R^is_*\to R^i\sigma_*B^*$, for $i\leq 1$, are isomorphisms.
We have that $\alpha\colon f^*W\to\bm1$ induces an isomorphism
$W\simeq f_*\bm1$.
\end{prop}
\begin{proof}
For every $A\in\clg{M}$, we have natural isomorphisms
\begin{align*}
\Hom(A,W) &\simeq \Hom(\bm1, W\otimes A^\vee) \\
&\simeq \Hom(\bm1, s_*(W\otimes A^\vee)) \\
&\simeq \Hom(\bm1, s_*(f_*\bm1\otimes A^\vee)) & \text{by \nref{Lemma}{lem:keylemmaforW}.}  \\
&\simeq \Hom(\bm1, f_*\bm1\otimes A^\vee) \\
&\simeq \Hom(A, f_*\bm1),
\end{align*}
and we conclude by the Yoneda lemma.
\end{proof}

\begin{rmk}[On choices and canonicity.]
Lots of choices are made in the constructions of $W$ and $U$.
Namely,
each time we take an extension corresponding to an extension class,
this involves a choice.
Moreover,
the successive lifts $\alpha_{n+1}$ of $\alpha_n$ giving the map $\alpha \colon f^*W \to \bm1$
all involve choices.
However,
the choice of $\alpha$ is in a sense the only one that matters:
after this,
the comparison isomorphism $B^*W \simeq U$ is uniquely determined,
in view of the universal properties
\nref{Prop.}{prop:characterizationU} and
\nref{Prop.}{prop:characterizationW}.
\end{rmk}

\subsection{Proof of Thm.~\ref{thm:U1}}\label{sec:unipotent-proof}
With the properties above established for $U$ and $W$, we can now prove the main result of this section.

\begin{proof}[Proof of \nref{Theorem}{thm:U1}]
We have that
\begin{align*}
B^*f_*\bm1
  \simeq B^*W
  \simeq U
  \simeq \varphi_*\bm1,
\end{align*}
by Propositions~\ref{prop:characterizationW},~\ref{prop:realization} and~\ref{prop:characterizationU}.
We claim that this isomorphism is the natural algebra morphism
$B^*f_*\bm 1 \to \varphi_*\bm1$
induced by the counit $b^*f^*f_*\bm1\to \bm1$.
This follows from the commutativity of the following diagram:
\begin{equation*}\begin{tikzcd}
\varphi^*B^*f_*\bm1 \ar[dd] \ar[rrr] \ar[dr]
  &[-1.5em]&&[-1.5em] \varphi^*B^*W \ar[dd] \ar[r]
  & \varphi^* U \ar[dd, "u"] \ar[r]
  & \varphi^*\varphi_*\bm1 \ar[dd, "\rom{coun}"] \\
& b^*f^*f_*\bm1 \ar[d, "b^*(\rom{coun})"] \ar[r]
  & b^*f^*W \ar[d, "b^*(\alpha)"] \ar[ur]\\
\bm1\ar[r,equals] & \bm1\ar[r,equals] & \bm1\ar[r,equals] & \bm1\ar[r,equals] & \bm1\ar[r,equals] & \bm1
\end{tikzcd}\end{equation*}
\end{proof}

\section{Malcev completions}\label{sec:malcev}
Let $\varrho\colon P\to K$ be a homomorphism of pro-algebraic groups.
\begin{defn}\label{defn:malcev_completion}
Let $\unip(\varrho^*)$ be the smallest Tannakian subcategory
(see
\nref{Def.}{defn:tannak_subcat})
of $\Rep P$
that contains the image of $\varrho^*$
and is stable under extensions.
The Tannaka dual of $\unip(\varrho^*)$ is called the \emph{Malcev completion} of $P$ with respect to $\varrho\colon P\to K$.
\end{defn}

\begin{lem}\label{lem:alt_malcev_defn}
We have that $\unip(\varrho^*)$ is
the
category
of $\Res^{\im(\varrho)}_P$-unipotent objects in $\Rep P$.
\end{lem}
\begin{proof}
By \nref{Lemma}{lem:im_of_res_from_im},
$\unip(\varrho^*)$ contains the image of $\Res^{\im(\varrho)}_P$.
Thus, it also contains all $\Res^{\im(\varrho)}_P$-unipotent objects,
since it's stable under extensions.
To conclude,
note that the class of $\Res^{\im(\varrho)}_P$-unipotent objects is stable under extension (\nref{Lemma}{lem:unip-ext}),
and apply \nref{Lemma}{lem:extn_closure}.
\end{proof}

\begin{lem}[Extension-closure of a Tannakian subcategory]\label{lem:extn_closure}
Let $\clg{T}$ be a neutral Tannakian category and
$\clg{S}\subseteq \clg{T}$ a Tannakian subcategory.
The smallest subcategory of $\clg{T}$
which contains $\clg{S}$ and is closed under extensions
is also a Tannakian subcategory.
\qed
\end{lem}

\begin{rmk}\label{rmk:malcev_completion_filtrations}
We spell things out more explicitly:
by Lemma~\ref{lem:alt_malcev_defn},
$\unip(\varrho^*)$ is the full subcategory of
$\Rep P$
with objects those $X \in \Rep P$
that admit a filtration
\begin{align*}
0 \subseteq X_0 \subseteq \dotsc \subseteq X_n = X
\end{align*}
such that $X_{i+1}/X_i$ lies in the image of $\Res^{\im(\varrho)}_P$.
Equivalently,
the $X_{i+1}/X_i$ are subquotients of objects in
the image of $\varrho^* = \Res^K_P$.
\end{rmk}

Lemma~\ref{lem:alt_malcev_defn} implies that,
replacing $K$ with $\im(\varrho)$,
we may assume that $\varrho$ is surjective,
without changing the Malcev completion.

\begin{rmk}[The universal property of the Malcev completion]\label{rmk:univ_prop_malcev}
Let $\clg{G}$ be the Malcev completion of $\varrho$.
We get an exact sequence of groups
\begin{align*}
1\to \clg{U}\to \clg{G}\to K,
\end{align*}
and the kernel $\clg{U}$ is pro-unipotent (\nref{Rmk.}{lem:unip_ker} and \nref{Lemma}{lem:alt_malcev_defn}).
The map $\clg{G}\to K$ is given dually, by sending a $K$-representation $X$
to $\varrho^*X$, which is an object in the full subcategory $\Rep\clg{G}$ of $\Rep P$.
There is a canonical map $\hat\varrho\colon P\to \clg{G}$ which by definition factors $\varrho$.
It's defined via $\hat\varrho^*$,
which is just the full inclusion of $\Rep\clg{G}$ into $\Rep P$.

The Malcev completion $\clg{G}$ is universal among groups factoring $\varrho$ in this way.
Let $G$ be another pro-algebraic group equipped with a map $G\to K$, which has a pro-unipotent kernel $U$,
and a map $\tilde\varrho\colon P\to G$ factoring $\varrho$.
\begin{equation*}\begin{tikzcd}
1\ar[r] &[-1em] U\ar[r] &[-1em] G \ar[r] &[-1em] K \\[.5em]
        &         & P \ar[u, "\tilde\varrho"]  \ar[ur, "\varrho"']
\end{tikzcd}\end{equation*}
There is then a unique map $\phi\colon\clg{G}\to G$
such that $\phi\circ\hat\varrho = \tilde\varrho$.
To see this,
we need to check that $\tilde{\varrho}^*$ lands in the full subcategory $\Rep\clg{G}$ of $\Rep P$.
Since $\Rep \clg{G}$ is stable under extensions and taking subquotients in $\Rep P$,
we only need to check that $\tilde{\varrho}^*$ lands in $\Rep \clg{G}$ when restricted
to those representations coming from $K$,
by \nref{Prop.}{prop:tannaka-mor-cor}.\ref{enum:surjective} and \nref{Lemma}{lem:unip_ker}.
But on those representations,
$\tilde{\varrho}^*$ coincides with $\varrho^*$,
which we know lands in $\Rep \clg{G}$,
and we are done.

\begin{equation*}\begin{tikzcd}
& P \ar[dl, "\hat\varrho"', bend right=16] \ar[dd, "\varrho" near end] \ar[dr, "\tilde\varrho", bend left=16] \\
\clg{G} \ar[dr, bend right=16] \ar[rr, "\phi" near end, dotted, crossing over] && G \ar[dl, bend left=16] \\
& K
\end{tikzcd}\end{equation*}
\end{rmk}

\begin{lem}\label{lem:malcev-surj}
Given a homomorphism $\varrho\colon P\to K$,
its Malcev completion $\hat\varrho\colon P\to\clg{G}$
is surjective.
\end{lem}
\begin{proof}
It follows immediately from \nref{Prop.}{prop:tannaka-mor-cor}.\ref{enum:surjective}.
\end{proof}

\begin{defn}\label{def:malcev-complete}
Let $P\to K$ be a homomorphism of pro-algebraic groups.
We say that it's \emph{Malcev complete} if the homomorphism
$\clg{G}\to K$ from the Malcev completion to $K$ is an isomorphism.
\end{defn}

\begin{exa}
As a trivial example, note that $\id\colon P\to P$,
and more generally any automorphism of $P$,
is always Malcev complete.
More generally, the Malcev completion of an injective homomorphism is simply the image of the homomorphism.
\end{exa}

\begin{lem}\label{lem:easy_malcev_complete}
A homomorphism $P\to K$ is Malcev complete if and only if
it's surjective
and the essential image of $\Rep K$ in $\Rep P$ is closed under extensions.
\qed
\end{lem}

\begin{lem}\label{lem:malcev_malcev}
Let $P \to G \to K$ be homomorphisms
such that $P \to G$ is Malcev complete.
Then,
the Malcev completions $\clg{G}$ and $\clg{G}'$
of $P \to K$ and $G \to K$,
respectively,
coincide.
\end{lem}
\begin{proof}
Note that $\Rep G$ is a full Tannakian subcategory of $\Rep P$,
and that $\Rep \clg{G}$ and $\Rep \clg{G}'$ are full Tannakian subcategories of
$P$ and $G$, respectively.
We check that $\Rep \clg{G}$ and $\Rep \clg{G}'$ coincide as collections of objects in $\Rep P$.
The former is constructed by successive extensions in $\Rep P$ of objects coming from $K$.
The latter is constructed by successive extensions in $\Rep G$ of objects coming from $K$.
But since $\Rep G$ is closed under extensions in $\Rep P$ by the previous lemma,
we are done.
\end{proof}

\begin{defn}
Let $P$ be a pro-algebraic group.
The \emph{unipotent completion} of $P$ is the Malcev completion of $P\to 1$.
\end{defn}

The unipotent completion of $P$ is the initial pro-unipotent group equipped with a homomorphism from $P$.
Moreover,
given a homomorphism $P\to K$,
the unipotent completion of $\ker(P \to K)$
is exactly the kernel $\clg{U}$ of the Malcev completion $\clg{G} \to K$.

\begin{rmk}\label{rmk:malcev_lit_diff}
Our context for defining Malcev completions differs a bit from what is often found in the literature.
Usually, $P$ is taken to be a discrete group $\pi$,
$K$ is taken to be (pro-)reductive,
and $\varrho\colon\pi\to K(F)$ is assumed to have Zariski dense image.
We round out this section by discussing these three differences,
as well as functoriality.
\end{rmk}

\begin{defn}
Given a discrete group $\pi$,
the Tannaka dual of $\Rep\pi$  is a pro-algebraic group.
We denote it by $\hat\pi$ and call it the \emph{pro-algebraic completion} of $\pi$.
(Here, $\Rep \pi$ is the category of finite-dimensional $F$-linear $\pi$-representations.)
\end{defn}

The pro-algebraic completion $\hat\pi$ comes equipped with a homomorphism $\pi\to\hat\pi(F)$.
Given a pro-algebraic group $K$,
composing homomorphisms $\hat\pi\to K$ with $\pi\to\hat\pi(F)$
gives a bijection $\Hom(\hat\pi,K) \simeq \Hom(\pi,K(F))$,
which is natural in $K$.
More precisely:
\begin{lem}\label{lem:proalg_adjoint}
The functor $\pi \mapsto \hat\pi$
from discrete groups to pro-algebraic groups
given by pro-algebraic completion is left adjoint to the functor
$P \mapsto P(F)$.
\qed
\end{lem}

Given a map from a discrete group $\pi$ to a pro-algebraic group $K$,
our definition of the Malcev completion of the corresponding map $\hat\pi \to K$,
coincides with definition of the Malcev completion of $\pi \to  K$
found in the literature.
This takes care of the first difference in \nref{Rmk.}{rmk:malcev_lit_diff}.

Next, let $\varrho\colon P\to K$ be a homomorphism as before.
As we've seen, replacing $K$ by $\im\varrho$,
we still get the same Malcev completion.
Therefore, we can assume that $\varrho$ is surjective
(or that the image is Zariski dense,
when $P$ is discrete).

Finally,
replacing $K$ with a quotient by a pro-unipotent normal subgroup,
still gives the same Malcev completion
(the most important example of this being the maximal (pro-)reductive quotient of $K$):
\begin{lem}\label{lem:malcev_red_quot}
Let $\varrho \colon P \to K$ be a homomorphism and
$q \colon K \twoheadrightarrow S$ a quotient by a pro-unipotent normal subgroup.
Then the Malcev completions of $P \to K$ and $P \to S$ coincide.
\end{lem}
\begin{proof}
We check it on the side of the Tannakian categories.
Trivially, $\unip(\varrho^*)$ contains $\unip(\varrho^*q^*)$
(both are full subcategories of $\Rep P$).
For the other direction,
note that every object in $\Rep K$ is obtained by
successive extension from objects coming from $q^*$
(\nref{Lemma}{lem:unip_ker}).
Thus, everything in the image of $\varrho^*$ is obtained by
successive extension from objects coming from $\varrho^*q^*$.
This gives the reverse inclusion,
using the extension-stability of $\unip(\varrho^*q^*)$ in $\Rep P$.
\end{proof}

\begin{rmk}
By \nref{Lemma}{lem:malcev-surj},
if $\varrho\colon P\to K$ is Malcev complete,
then $\varrho$ is surjective.
Moreover, the group $K$ and the homomorphism $P\to K$
are determined by the maximal pro-reductive quotient $S$ of $K$
and the homomorphism $P\to S$.
We recover $K$ as the Malcev completion of $P\to S$.
\end{rmk}

\begin{prop}[Functoriality]\label{prop:malcev_functoriality}
Consider a commutative diagram of pro-algebraic groups,
\begin{equation*}\begin{tikzcd}
K \ar[rr] &[-1em]&[-1em] K' \ar[from=dd] \\[-1em]
& \clg{G} \ar[ul] \ar[rr, dotted, crossing over] &&[-1em] \clg{G}' \ar[ul] \\[-1em]
P \ar[uu] \ar[ur] \ar[rr] && P' \ar[ur]
\end{tikzcd}\end{equation*}
where $\clg{G}$ and $\clg{G}'$ are the Malcev completions
of $P \to K$ and $P' \to K'$, respectively.
Then, there is a unique group homomorphism
$\clg{G} \to \clg{G'}$
making everything commute.
Moreover, this construction is functorial.
\end{prop}
\begin{proof}
To give such a map
is the same as to give a map
$\clg{G} \to \clg{G}' \times_{K'} K$,
which
when composed with the projection to $K$
gives the canonical map $\clg{G} \to K$,
and which when precomposed with $P \to \clg{G}$
gives the obvious arrow $P \to \clg{G}' \times_{K'} K$.
But the kernel
\begin{align*}
\ker\left( \clg{G}' \times_{K'} K \to K \right)
  = \ker\left( \clg{G}' \to K' \right)
\end{align*}
is pro-unipotent,
so the universal property of $\clg{G}$ gives us a unique map
of the form we wanted.
Functoriality follows easily from the uniqueness.
\end{proof}

\begin{prop}\label{prop:malcev_adjoint}
Malcev completion gives a functor from the category of homomorphisms of pro-algebraic groups,
to the category of homomorphisms of pro-algebraic groups with pro-unipotent kernels.
Moreover, this functor is left adjoint to the inclusion functor.
\end{prop}
\begin{proof}
We already proved the functoriality,
what is left to show is the adjunction.
This follows easily from the universal property of Malcev completion,
and the functoriality above (including the uniqueness statement).
\end{proof}

\section{The Malcev completeness criterion}\label{sec:thmM}
We once again place ourselves in the Standard Situation,
keeping all the notation from \nref{Section}{sec:standard}:
\begin{equation*}\begin{tikzcd}
\clg{M} \ar[d, "f^*" '] \ar[r, "B^*"] & \clg{T} \ar[d, "\varphi^*" '] \\
\clg{E} \ar[u, bend right, "s^*" '] \ar[r, "b^*" '] & \VecOp \ar[u, bend right, "\sigma^*" ']
\end{tikzcd}
\qquad\quad
\begin{tikzcd}
1\ar[r] & K\ar[r,"\iota"]              & G\ar[r,"s"'] & H \ar[l, bend right, "f"'] \ar[r] & 1 \\
        & P \ar[u,"\varrho"] \ar[ur, bend right=24, "B"']
\end{tikzcd}\end{equation*}

Let us introduce some more notation that will be helpful in discussing the
Malcev completeness of $\varrho$.
Fix a quotient $q_1\colon K\twoheadrightarrow S$ with a pro-unipotent kernel.
Denote the composition of $\varrho$ with $q_1$ by
$\bar\varrho\colon P\to S$.
We've seen, in \nref{Lemma}{lem:malcev_red_quot},
that the Malcev completions of $\varrho$ and $\bar\varrho$ coincide.

Let $\mlc$ be the full subcategory $\unip(\bar\varrho^*)$ of $\clg{T}$,
as in \nref{Def.}{defn:malcev_completion}.
Its objects are those built by successive extension from objects in the image of $\bar\varrho^*$
(see Rmk~\ref{rmk:malcev_completion_filtrations}).
In the following diagram, I omit the (systematically named) sections
($s^*,t^*,\sigma^*$ of $f^*,g^*,\varphi^*$, plus variants with tilde)
in order to reduce clutter.
Let $\widetilde{\clg{M}}$ be the (full) preimage of $\Rep S$ in $\clg{M}$.
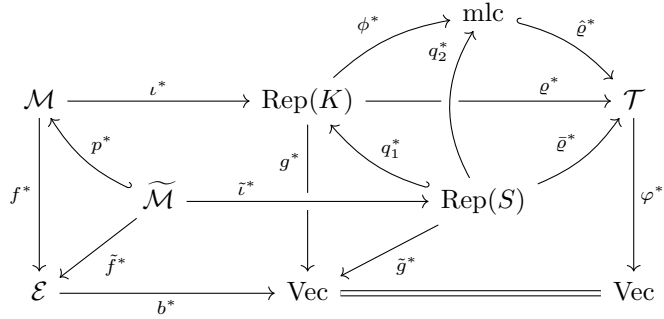
\begin{equation*}\begin{tikzcd}
&&& \mlc \ar[dr, "\hat\varrho^*", hook, bend left=16]
\\
\clg{M} \ar[dd, "f^*" '] \ar[rr, "\iota^*"]
  && \Rep(K)
    \ar[ur, "\phi^*", bend left=16]
    \ar[dd, "g^*" ' near start]
    \ar[rr, "\varrho^*" near end]
  && \clg{T} \ar[dd, "\varphi^*"]
\\
& \widetilde{\clg{M}}
    \ar[ul, hook', bend left=16, "p^*"']
    \ar[rr, "\tilde\iota^*" near start, crossing over]
    \ar[dl, "\tilde f^*"]
  && \Rep(S)
    \ar[uu, bend left, "q_2^*" near end, crossing over]
    \ar[ul, hook', bend left=16, "q_1^*" ']
    \ar[ur, "\bar\varrho^*", bend right=16]
    \ar[dl, "\tilde g^*"]
\\
\clg{E} \ar[rr, "b^*" ']
  && \VecOp \ar[rr, equal]
  && \VecOp
\end{tikzcd}\end{equation*}

Writing $\clg{K}$ for the Malcev completion of $\bar\varrho\colon P\to S$,
i.e., the Tannaka dual $\scr{G}(\mlc)$ of $\mlc$,
we can formulate our goal as the following theorem:
\begin{thm}[Malcev completeness criterion]\label{thm:M}
Assume that $\bar\varrho$ is surjective and that
the canonical maps $b^*R^is_* \to R^i\sigma_*B^*$, for $i=0,1$, are isomorphisms.
Then $\phi\colon\clg{K}\to K$ is an isomorphism
(or equivalently, $\varrho\colon P \to K$ is Malcev complete).
\end{thm}

\begin{rmk}
This is the main technical result of this paper,
and it has two key assumptions:
the surjectivity of $\bar\varrho$ we refer to as the \emph{surjectivity assumption},
and the other assumption we call the \emph{cohomological assumption}.
Note that the conclusion of the theorem implies the surjectivity of $\varrho$.
\end{rmk}

Our general strategy is, as it was for
Thm.~\ref{thm:U2},
to utilize \nref{Prop.}{prop:group-iso}.
That means we construct models of the regular representations of the groups and
show that the natural map between them is an isomorphism.

Since both $K$ and $\clg{K}$ are extensions of $S$ by unipotent groups,
it's not too surprising that the unipotent theorem (Thm.~\ref{thm:U2}) becomes useful.
In \nref{Section}{sec:unipotent-module-cats},
we show how, given a relatively unipotent neutral Tannakian category,
one can produce a new one,
which is unipotent over a deeper base
(\nref{Prop.}{prop:module}).
In the rest of \nref{Section}{sec:thmM}, we use this result to apply the unipotent theorem,
and finally prove Thm.~\ref{thm:M}.

\subsection{On certain tensor products of Tannakian categories}\label{sec:unipotent-module-cats}
Consider a 2-commutative diagram
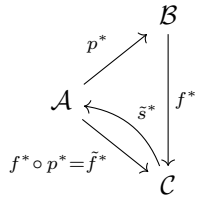
\begin{equation*}\begin{tikzcd}
& \clg{B} \ar[dd, "f^*"] \\
\clg{A} \ar[dr, "f^*\circ\, p^*=\tilde f^*" '] \ar[ur, "p^*"] \\
& \clg{C} \ar[ul, bend right, "\tilde{s}^*"']
\end{tikzcd}\end{equation*}
in $\nTan$.
We assume that $\tilde{s}^*$ is a monoidal 2-section of $\tilde{f}^*$.
This gives a monoidal 2-section $s^* := p^*\tilde{s}^*$ of $f^*$.
Before we move on to the main topic of this subsection,
we state a basic lemma which becomes useful in the sequel.

As in \nref{Section}{sec:group-iso},
we have an algebra morphism
$p_\rom{alg}\colon p^*\tilde{f}_*\bm1\to f_*\bm1$.
It is uniquely characterized
by a compatibility with Hopf algebra counits after applying $f^*$.
Under adjunction, it corresponds to a morphism
$\tilde{f}_*\bm1\to p_*f_*\bm1$,
and this morphism is also a morphism of algebras.
On the other hand,
there is an isomorphism
$\tilde{f}_*\bm1\simeq p_*f_*\bm1$
of objects
in $\Ind\clg{A}$,
induced by the $2$-commutativity of the above diagram.

\begin{lem}\label{lem:double-right-adjoint-alg-iso}
The two morphisms
$\tilde{f}_*\bm1\to p_*f_*\bm1$,
described above,
coincide.
\qed
\end{lem}

Let $A:=p^*(\tilde f_*\bm1)$: it's an algebra  object in $\Ind\clg{B}$.
We consider the tensor category $\Modfg(A)$ of finitely generated $A$-modules in $\Ind\clg{B}$,
as in \nref{Section}{sec:modules},
sitting as a full subcategory in the abelian tensor category $\Mod(A)$ of $A$-modules in $\Ind\clg{B}$.

As in \nref{Section}{sec:standard}, it's a fact that $f^*A = \tilde f^*\tilde f_*\bm1$ is a Hopf algebra object in $\clg{C}$,
which thus has a counit map $f^*A\to \bm1$.
We may therefore define a monoidal, right-exact functor $w^*\colon\Modfg(A)\to \clg{C}$
by $w^*(\blank) := \bm1\otimes_{f^*A}f^*(\blank)$.
(A finitely generated $\bm1$-module in $\Ind\clg{C}$ is just an object of $\clg{C}$.)
It has a monoidal\footnote{%
  In particular, the isomorphism $w^*e^*\simeq\id$ is monoidal.}
2-section $e^*$ defined by $e^*(\blank) := A\otimes s^*(\blank)$.

\begin{prop}\label{prop:module}
Assume $\clg{B}$ is $p^*$-unipotent,
that $p^*$ is fully faithful,
and that the essential image of $p^*$ is closed under taking subobjects.
Then $\Modfg(A)$ is
\emph{(i)} $e^*$-unipotent and
\emph{(ii)} neutral Tannakian.
Moreover, $w^*$ is faithful and exact.
\end{prop}

\begin{rmk}
In fact, $\Modfg(A)$ is the tensor product
$\clg{B}\otimes_\clg{A}\clg{C}$ of $\clg{B}$ and $\clg{C}$ over $\clg{A}$,
once one has defined what that means.
Group theoretically,
it means that $\scr{G}(\Modfg(A))$ is isomorphic to
$\scr{G}(\clg{B}) \times_{\scr{G}(\clg{A})} \scr{G}(\clg{C})$,
and in particular,
that $\scr{G}(\clg{B}) \to \scr{G}(\clg{A})$
and $\scr{G}(\Modfg(A)) \to \scr{G}(\clg{C})$ have the same (pro-unipotent) kernel.
\end{rmk}

We'll prove this proposition using the following result due to Deligne:
\begin{lem}[{\cite[Cor.~2.10]{deligne90}}]\label{lem:deligne-cor-2.10}
Let $\clg{T}$ be a rigid abelian tensor category satisfying $\End\bm1 = F$,
and let $\clg{M}$ be a non-zero abelian monoidal category with a right exact tensor product.
Then any right exact functor $\clg{T}\to\clg{M}$ is left exact and faithful.
\end{lem}

The above lemma gives, in particular,
that both $e^*\colon\clg{C}\to\Modfg(A)$
and $A\otimes(\blank)\colon\clg{B}\to\Modfg(A)$
are exact,
which is useful.
We now spend some time establishing the assumptions of Deligne's result for the functor $w^*$.
More precisely,
we prove that $\Modfg(A)$ is rigid and abelian, and that $\End\bm1 = F$.
To prove \nref{Prop.}{prop:module},
it is then essentially enough to show the $e^*$-unipotency.

\begin{lem}\label{lem:fg-noetherian}
Assume that $\clg{B}$ is $p^*$-unipotent,
that $p^*$ is fully faithful,
and that the essential image of $p^*$ is closed under taking subobjects.
Then $\Modfg(A)$ is Noetherian,
i.e., any ascending chain of subobjects of an object in $\Modfg(A)$ must stabilize.
\end{lem}
\begin{proof}
Since a quotient of a Noetherian object is Noetherian,
it's enough to consider finitely generated free modules $A\otimes Y$.
We proceed by induction on the $p^*$-unipotency index of $Y$.

\emph{Base case:}
We have that
$A\otimes p^*X \simeq p^*(\tilde{f}_*\bm1 \otimes X)$,
so we reduce to studying sub-$\tilde{f}_*\bm1$-modules of $\tilde{f}_*\bm1 \otimes X$,
by the assumption on the essential image of $p^*$.
But $\Mod(\tilde{f}_*\bm1)$ is equivalent to $\clg{C}$
(e.g., by the fundamental theorem of Hopf modules and Tannakian duality),
and $\clg{C}$ is Noetherian (since it's Tannakian).

\emph{Induction step:}
We assume that $Y$ is an extension
of $p^*X$ by $Y'$
such that $A\otimes Y'$ is a Noetherian $A$-module.
By the exactness of $A\otimes(\blank)$,
we get $A\otimes Y$ as an extension of $A\otimes p^*A$ by $A\otimes Y'$.
A submodule $N$ of $A\otimes Y$ gives rise to
a submodule $N'$ of $A\otimes Y'$
by intersecting with $A\otimes Y'$ (pullback),
and to a submodule $N''$ of $A\otimes p^*X$ by taking the quotient by $N'$.
If a larger submodule $\tilde N\supseteq N$ gives rise, in this way,
to the same $N'$ and $N''$,
then it must coincide with $N$.
Since $A\otimes Y'$ and $A\otimes p^*X$ are Noetherian,
we are done.
\end{proof}

\begin{lem}\label{lem:fg-abelian}
Assume that all finitely generated free $A$-modules are Noetherian.
Then, $\Modfg(A)$ is closed under taking subobjects in $\Mod(A)$.
In particular, $\Modfg(A)$ is abelian.
\end{lem}
\begin{proof}
Let $M$ be an $A$-module which is finitely generated by $Y\in \clg{B}$,
and let $N$ be a submodule of $M$.
Then $N$ is a quotient of a submodule $N'$ of $A\otimes Y$
(namely the pullback of $N$ in $A\otimes Y$),
and we can thus replace $M$ by $A\otimes Y$,
without loss of generality.

We now show that for every $Y\in\clg{B}$,
every sub-$A$-module of $A\otimes Y$ is finitely generated.
Let $N$ be such a submodule, generated by $\tilde Y$ in $\Ind\clg{B}$,
i.e., we have a map
$A\otimes\tilde{Y}\twoheadrightarrow N$.
Let the kernel of this map be denoted by $K$,
and let $(\tilde Y_\alpha)_{\alpha\in I}$ be a filtered system of $\tilde Y_\alpha\in\clg{B}$
such that $\tilde{Y} = \colim_\alpha \tilde Y_\alpha$.
We may assume that the $\tilde Y_\alpha$ are subobjects of $\tilde{Y}$
and that all transition maps are monomorphisms~\cite[Lemme~4.2.1]{deligne-P1-3pts}.
Note that $A\otimes \tilde{Y}$ is then isomorphic to $\colim_\alpha A\otimes \tilde Y_\alpha$,
a filtered colimit of sub-$A$-modules,
by the exactness of $A\otimes(\blank)$.

Let $K_\alpha$ be the intersection of $K$ and $A\otimes \tilde Y_\alpha$ in $A\otimes\tilde{Y}$,
and let $N_\alpha$ be the cokernel of $K_\alpha\to A\otimes \tilde Y_\alpha$.
We have short exact sequences
\begin{align*}
0\to K_\alpha\to A\otimes \tilde Y_\alpha\to N_\alpha\to 0,
\end{align*}
and
since colimits commute with colimits,
we get that $\colim_\alpha N_\alpha$ is the cokernel of $K\to A\otimes\tilde{Y}$,
i.e., that it's $N$.
The Noetherianity of $A\otimes Y$
implies that the partially ordered system $N_\alpha$
of subobjects of $N$ has a maximal element $N_\beta$ (since they are also subobjects of $A\otimes Y$).
The fact that it's filtered implies that there is a unique such maximal subobject,
and it must be equal to all of $N$,
since the colimit equals $N$.
All in all, we get
\begin{align*}
N = N_\beta = (A\otimes \tilde Y_\beta)/K_\beta,
\end{align*}
so that $N$ is finitely generated by $\tilde Y_\beta$,
and we're done.
\end{proof}

\begin{lem}\label{lem:special-case-proj}
We have that $A\otimes p^*X \simeq e^*\tilde{f}^*X$.
\end{lem}
\begin{proof}
This follows from the projection formula,
since
\begin{align*}
A\otimes p^*X = p^*\tilde{f}_*\bm1 \otimes p^*X
  \simeq p^*(\tilde{f}_*\bm1\otimes X),
\end{align*}
and then, by \nref{Lemma}{lem:tensor-with-reg-rep},
\begin{align*}
A\otimes p^*X \simeq p^*(\tilde{f}_*\bm1\otimes \tilde{s}^*\tilde{f}^*X)
  \simeq A\otimes s^*\tilde{f}^*X
  \simeq e^*\tilde{f}^*X.
\end{align*}
\end{proof}

\begin{lem}\label{lem:module-base-case}
Assume that $\clg{B}$ is $p^*$-unipotent,
that $p^*$ is fully faithful,
and that the essential image of $p^*$ is closed under taking subobjects.
Let $M$ be a quotient of $e^*Z$ in $\Modfg(A)$.
Then $M$ is in the essential image of $e^*$.
\end{lem}
\begin{proof}
First, replace $Z$ by $Z/Z_0$,
where $Z_0$ is the largest subobject of $Z$ such that $e^*Z_0 \to M$ is zero.
By Lemmas~\ref{lem:fg-noetherian} and~\ref{lem:fg-abelian},
the kernel $K$ of $e^*Z \to M$ is finitely generated,
i.e., it's the quotient of some $A\otimes Y$.
If $Y$ is non-zero,
the $p^*$-unipotency of $\clg{B}$ implies that $Y$ has a non-zero subobject
of the form $p^*X$.
But $A\otimes p^*X \simeq e^*\tilde{f}^*X =: e^*Z'$,
by \nref{Lemma}{lem:special-case-proj}.
We thus have a map $e^*Z' \to e^*Z$,
and its composition with $e^*Z \to M$ is zero.
By the fullness of $e^*$,
we have an underlying map $Z' \to Z$,
and it must be zero,
because of the simplification of $Z$ we started with.
Therefore,
we may replace $Y$ by $Y/e^*Z$,
and in the end assume $Y$ is zero,
so that $K$ is zero and $e^*Z \simeq M$.
\end{proof}

\begin{lem}\label{lem:endos}
$\End_{\Mod(A)}(A) \simeq F$.
\end{lem}
\begin{proof}
Since $p^*$ is fully faithful, we have
$\End_{\Mod(A)}(A) \simeq \End_{\Mod(\tilde f_*\bm1)}(\tilde f_*\bm1)$,
and the latter is $F$:
\begin{align*}
\Hom_{\Mod(\tilde{f}_*\bm1)}(\tilde{f}_*\bm1\otimes\bm1, \tilde{f}_*\bm1)
  \simeq \Hom_{\Ind\clg{A}}(\bm1, \tilde{f}_*\bm1)
  \simeq \Hom_{\clg{C}}(\bm1,\bm1)
  \simeq F,
\end{align*}
under the adjunction $\tilde{f}_*\bm1\otimes(\blank)\colon \Ind\clg{A} \rightleftarrows \Mod(\tilde{f}_*\bm1) \,:\!\rom{Forgetful}$.
\end{proof}

We're now ready to prove the main result of this section.
\begin{proof}[Proof of \nref{Proposition}{prop:module}]
(i)
Take $N$ in $\Modfg(A)$,
a quotient of $A\otimes Y$ for some $Y$ in $\clg{B}$.
We want to establish that $N$ is
an extension of $e^*Z$ by something $e^*$-unipotent,
for some $Z\in\clg{C}$.
This is done by induction on the $p^*$-unipotency index of $Y$.

\emph{Base case:}
Assume $Y=p^*X$.
Then
$A\otimes Y = e^*(\tilde{f}^*X)$,
by \nref{Lemma}{lem:special-case-proj}.
Then \nref{Lemma}{lem:module-base-case}, says that $N$ must be of the form $A\otimes s^*(Z) = e^*(Z)$,
for some $Z\in\clg{C}$.
In particular, it's $e^*$-unipotent.

\emph{Induction step:}
We write $Y$ as an extension
\begin{align*}
0\to Y'\to Y\to p^*X \to 0,
\end{align*}
where $Y'$ is of a lower $p^*$-unipotency index.
The induction hypothesis tells us that any quotient of $A\otimes Y'$ in $\Modfg(A)$ is $e^*$-unipotent.
Thus,
\begin{equation*}\begin{tikzcd}
0 \ar[r] & A\otimes Y' \ar[d, dotted, two heads] \ar[r]
  & A\otimes Y \ar[d, two heads] \ar[r]
  & A\otimes p^*(X) \ar[d, dotted, two heads] \ar[r] & 0 \\
0 \ar[r] & U \ar[r] & N \ar[r] & e^*(Z) \ar[r] & 0
\end{tikzcd}\end{equation*}
where $U$ is $e^*$-unipotent, and we're done.

(ii)
With Lemmas~\ref{lem:fg-abelian} and~\ref{lem:endos} in mind,
it's enough to show that
(a) $\Modfg(A)$ is rigid,
and (b) $w^*$ is faithful and exact.
The existence of a fibre functor then follows from (b) and the neutrality of $\clg{C}$.
(See \nref{Section}{sec:tannaka}.)

(a)
Rigidity is stable under extension,
and everything in the image of $e^*$ is rigid,
since $\clg{C}$ is rigid and $e^*$ is monoidal.

(b)
Clearly, $w^*$ is right exact.
By \nref{Lemma}{lem:deligne-cor-2.10},
it is therefore exact and faithful.
\end{proof}

\subsection{Proof of Thm.~\ref{thm:M}}
We now go back to the situation described in the introduction to \nref{Section}{sec:thmM}.

\subsubsection{Applying the unipotent theorem}
We write $\hat\varphi^* := \varphi^*\hat\varrho^*\colon\mlc\to\VecOp$.
We also write $\hat\sigma^*$ for a section to $\hat\varphi^*$
such that $\hat\varrho^*\hat\sigma^*=\sigma^*$.

Apply \nref{Prop.}{prop:module} to $(\Rep S, \mlc, \VecOp)$ as $(\clg{A},\clg{B},\clg{C})$,
and write $T$ for $q_2^*\tilde{g}_*\bm1$ (playing the role of the algebra $A$).
The unipotency assumption needed follows from
the pro-unipotency of the kernel of $\clg{K} \to S$
and the surjectivity of $\bar{\varrho}$
(by \nref{Lemma}{lem:unip_ker}).
The proposition then gives us a unipotent Tannakian category
\begin{equation*}\begin{tikzcd}
\Modfg(T) \ar[d, "\omega^*"'] \\
\VecOp \ar[u, bend right, "\varepsilon^*"']
\end{tikzcd}\end{equation*}
with a forgetful functor to $\Ind\mlc$.

\begin{rmk}
In fact,
$\Mod^\rom{fg}(T)$
is nothing but the category
of representations of the (pro-unipotent)
kernel of $\clg{K} \twoheadrightarrow S$.
\end{rmk}

It is not hard to see that $\clg{M}$ is unipotent over $\widetilde{\clg{M}}$.
In fact, $\scr{G}(\widetilde{\clg{M}})$ is the pushout of $G \leftarrow K \to S$,
so that $G \to \scr{G}(\widetilde{\clg{M}})$ is surjective
with the same pro-unipotent kernel as $K \to S$.
We can thus apply \nref{Prop.}{prop:module} to $(\widetilde{\clg{M}},\clg{M},\clg{E})$
as $(\clg{A},\clg{B},\clg{C})$,
with $M := p^*\tilde{f}_*\bm1$ playing the role of $A$.
We get a relatively unipotent Tannakian category
\begin{equation*}\begin{tikzcd}
\Modfg(M) \ar[d, "w^*"'] \\
\clg{E} \ar[u, bend right, "e^*"']
\end{tikzcd}\end{equation*}
over $\clg{E}$, which has a forgetful functor to $\Ind\clg{M}$.

\begin{rmk}
In fact,
$\Mod^\rom{fg}(M)$
is nothing but $\Rep(\ker(K \to S) \ltimes H)$.
\end{rmk}

The composition of $\iota^*$ and $\phi^*$ induces a functor
$\Modfg(M)\to\Modfg(T)$,
since
\begin{align*}
\phi^*\iota^*M = \phi^*\iota^*p^*\tilde{f}_*\bm1
  \simeq q_2^*\tilde\iota^*\tilde{f}_*\bm1
  \simeq T.
\end{align*}
Here, we've used the fact that $\tilde\iota^*\tilde{f}_*\bm \simeq \tilde{g}_*\bm1$,
which is true by definition.
To summarize the situation,
we have a diagram
\begin{equation}\begin{tikzcd}\label{eq:module_unip_situation}
\Modfg(M) \ar[d,"w^*"'] \ar[r, "\phi^*\iota^*"]
  &[2em] \Modfg(T) \ar[d, "\omega^*"'] \\
\clg{E} \ar[u, bend right, "e^*"'] \ar[r, "b^*"]
  & \VecOp \ar[u, bend right, "\varepsilon^*"']
\end{tikzcd}\end{equation}
in $\nTan$,
where $\Modfg(M)$ is relatively unipotent over $\clg{E}$ and $\Modfg(T)$ is unipotent.
In order to apply the unipotent theorem,
we need the following result:
\begin{lem}
If
$b^*\circ R^is_* \simeq R^i\sigma_*\circ B^*$ for $i\leq 1$,
then
$b^*\circ R^ie_* \simeq R^i\varepsilon_* \circ(\phi^*\iota^*)$ for $i\leq 1$.
\end{lem}
\begin{proof}
Note that $R\varepsilon_*$ is the right adjoint of $\varepsilon^*$ (on the derived level, of course).
By definition, $\varepsilon^*$ is the composition $(T\otimes\blank)\circ\hat\sigma^*$.
The (ind-)right adjoints of $(T\otimes\blank)$ and $\hat\sigma^*$ are $\rom{Forget}\colon\Modfg(T)\to\Ind\mlc$, and $\hat\sigma_*$, respectively.
The former is exact and takes injectives to injectives (since $(T\otimes\blank)$ is exact),
and we get $R\varepsilon_* = (R\hat\sigma_*)\circ\rom{Forget}$.
In the same way, $Re_* = (Rs_*)\circ\rom{Forget}'$,
where $\rom{Forget}'\colon\Modfg(M)\to\Ind\clg{M}$.

Thus for $i\leq 1$,
\begin{align*}
b^*\circ(R^ie_*)
  &= b^*\circ(R^is_*)\circ\rom{Forget}' \\
  &= (R^i\sigma_*)\circ B^*\circ \rom{Forget}' \\
  &= (R^i\sigma_*)\circ (\hat\varrho^*\phi^*\iota^*)\circ \rom{Forget}' \\
  &= (R^i\sigma_*)\circ \hat\varrho^*\circ \rom{Forget}\circ (\phi^*\iota^*),
\end{align*}
and we've reduced to showing $(R^i\sigma_*)\circ \hat\varrho^* = R^i\hat\sigma_*$.
For $i=0$,
this follows from the surjectivity of $\hat\varrho$.
For $i=1$,
it follows from the stability of $\mlc$ in $\clg{T}$ under extensions.
\end{proof}

We may thus apply
\nref{Thm.}{thm:U1}
to the diagram~\heqref{eq:module_unip_situation},
yielding the following theorem:
\begin{thm}\label{thm:MU}
Assume that
$\bar\varrho$ is surjective and that
$b^* R^is_* \to R^i\sigma_* B^*$
is an isomorphism for $i=0,1$.
Then the natural algebra morphism
$\phi^*\iota^*w_*\bm1 \to \omega_*\bm1$ in $\Modfg(T)$
is an isomorphism.
\qed
\end{thm}

\subsubsection{The final argument}
\begin{proof}[Proof of Theorem~\ref{thm:M}]
From the previous section,
we get that
\begin{align*}
\rom{Forget}(\phi^*\iota^*w_*\bm1) \to \rom{Forget}(\omega_*\bm1)
\end{align*}
in $\Ind\mlc$ is an isomorphism of algebras (\nref{Thm.}{thm:MU}).
Note that, writing
$\rom{Forget}''\colon\Modfg(\iota^*M)\to \Ind\Rep K$,
we have
\begin{align*}
\rom{Forget}(\phi^*\iota^*w_*\bm1)
  = \phi^*\circ\rom{Forget}''\circ\iota^*w_*\bm1
  = \phi^*\iota^*\circ\rom{Forget}'\circ w_*\bm1
\end{align*}
\nref{Lemma}{lem:double-right-adjoint-alg-iso} gives us algebra isomorphisms
$f_*\bm1 \simeq \rom{Forget'}\circ w_*\bm1$
(note that $M\otimes(\blank)$ is a morphism in $\nTan$, by \nref{Lemma}{lem:deligne-cor-2.10},
and that $\rom{Forget}'$ is its ind-right adjoint),
and $\hat\varphi_*\bm1 \simeq \rom{Forget}\circ\omega_*\bm1$.
Finally, $\iota^*f_*\bm1 = g_*\bm1$, by definition.
Putting everything together
we thus get that the natural algebra morphism
\begin{align*}
\phi^*g_*\bm1 \to \hat\varphi_*\bm1
\end{align*}
is an isomorphism.
By \nref{Prop.}{prop:group-iso},
we are done.
\end{proof}

\newpage
\part{Fundamental groups}\label{part:applications}
\bigskip\bigskip

\section{Enriched local systems}\label{sec:enriched}
In this section,
we work out,
in an axiomatic setting,
the common arguments in our two main applications of the Malcev completeness criterion (Thm.~\ref{thm:M}).
The applications themselves are found in Section~\ref{sec:applications}.
The main results of this section are
Propositions~\ref{prop:enriched_coh_assump},~\ref{prop:enriched_surjectivity},
and~\ref{prop:generalized_hain_enriched}.

Let $k$ be a fixed subfield of $\C$.
We usually write $\pt$ for $\Spec k$.
For a $k$-variety $X$,
let $X^\rom{an}$ denote $X(\C)$ equipped with the complex analytic topology,
and let $\Loc(X)$ denote the category of
locally constant sheaves of finite-dimensional $F$-vector spaces (\emph{local systems}) on $X^\rom{an}$.
Given a $k$-point $x$ of $X$,
denote by $\pi_1(X,x)$ the (topological) fundamental group of $X^\rom{an}$.
Recall that $\Loc(X)$ is equivalent to
the category of (finite-dimensional) $F$-representations of $\pi_1(X,x)$.
Lastly,
$\SmVar_k$ is the category of
smooth, separated and geometrically connected schemes of finite type over $k$.

\begin{defn}\label{defn:enriched}
Following~\cite{arapura-hodge},
we say that a~\emph{weak theory of enriched local systems}
is
a contravariant 2-functor
\begin{align*}
\clg{E}\colon (\SmVar_k)^\rom{op} &\to \nTan         \\
  X                    &\mapsto\clg{E}(X) \\
  (f\colon Y\to X)     &\mapsto (f_\clg{E}^*\colon\clg{E}(X)\to\clg{E}(Y))
\end{align*}
together with a monoidal natural transformation $b^* = b^*_\clg{E }\colon\clg{E}(\blank)\to \Loc(\blank)$.
(Recall that the morphisms in $\nTan$ are faithful and exact functors.)
\end{defn}

For any $X$ in $\SmVar_k$ with a $k$-point $x$,
we then get an instance of the Standard Situation (\nref{Section}{sec:standard}),
and with that in mind,
we denote by $K = K(X,x)$ is the kernel of $s\colon \scr{G}(\clg{E}(X)) \to \scr{G}(\clg{E}(\pt))$.
As in Section~\ref{sec:standard},
we get a commutative diagram:
\begin{equation*}
\begin{tikzcd}
\clg{E}(X) \ar[d, "f^*" '] \ar[r, "\iota^*"] \ar[rr, bend left, "b_X^*"]
  & \Rep K \ar[d, "g^*"'] \ar[r, "\varrho^*"]
  & \Loc(X) \ar[d, "\varphi^*" '] \\
\clg{E}(\pt) \ar[u, bend right, "s^*"'] \ar[r, "b^*"'] & \VecOp \ar[u, bend right, "t^*"'] \ar[r, equals] & \VecOp \ar[u, bend right, "\sigma^*"']
\end{tikzcd}
\end{equation*}
Here, $s^*=p_\clg{E}^*$ and $\sigma^*=p_{\Loc}^*$ are induced by the structure morphism $p\colon X\to\pt$,
while $f^*=x_\clg{E}^*$ and $\varphi^*=x_{\Loc}^*$ are induced by the point $x\colon\pt\to X$.
We briefly translate the current setting to group theory,
starting with the following definitions:
\begin{defn}\label{defn:E-fund}
We denote by $\Loc_{\clg{E}}(X)$ the smallest Tannakian subcategory of $\Loc(X)$ containing the image of $b^*$.
We define the $\clg{E}$-\emph{fundamental group} of $(X,x)$ to be the Tannaka dual group
\begin{align*}
\pi_1^{\clg{E}}(X,x) := \scr{G}(\Loc_{\clg{E}}(X), x^*)
\end{align*}
of $\Loc_{\clg{E}}(X)$, with fibre functor $x^*=\varphi^*$.
The objects of $\Loc_{\clg{E}}(X)$ are local systems \emph{coming from $\clg{E}$},
or alternatively, \emph{of $\clg{E}$-origin}.
Lastly, let $G(X,x) := \scr{G}(\clg{E}(X), x^* \circ b^*)$
and $G(k) := \scr{G}(\clg{E}(\pt), b^*)$,
and let $S(X,x)$ be the maximal pro-reductive quotient of $K(X,x)$.
\end{defn}

Note that the map $\pi_1(X,x) \to G(X,x)$
induced by $b^*\colon \Rep G(X,x) \to \Rep \pi_1(X,x)$
factors through a map $\varrho \colon \pi_1(X,x) \to K(X,x)$,
as above.
By definition, $\varrho\colon \pi_1(X,x) \to K(X,x)$ factors further through
$\pi_1^{\clg{E}}(X,x)$,
and in fact the latter is the Zariski closure of the image $\im(\varrho)$ of the former
(by \nref{Lemma}{lem:im_of_res_from_im}).
In the end, we get a diagram
\begin{equation*}\begin{tikzcd}
&[-1em] S(X,x) \\[-0.5em]
1 \ar[r] & K(X,x) \ar[u, two heads] \ar[r]
  &[-1em] G(X,x) \ar[r]
  & G(k) \ar[l, bend right, start anchor = north west, end anchor = north east, yshift=-.5em] \ar[r] & 1 \\[-0.5em]
\pi_1(X,x) \ar[ur, bend right=4] \ar[r] & \pi_1^\clg{E}(X, x) \ar[u, hook]
\end{tikzcd}\end{equation*}
where the middle row is (split) exact.

Above, whenever we have a map from the discrete group $\pi_1(X,x)$
to a pro-algebraic group $G$,
we mean a map to the $F$-points $G(F)$ of $G$
(or equivalently a map from the pro-algebraic completion of $\pi_1(X,x)$ to $G$).

The goal of this section is to develop tools to show that
$\pi_1^\clg{E}(X,x) = K(X,x)$
and that
$\pi_1(X,x) \to \pi_1^\clg{E}(X,x)$
is Malcev complete,
both of which follow from simply showing that $\pi_1(X,x) \to K(X,x)$ is Malcev complete.
To show this,
we utilize the Malcev completeness criterion (Thm.~\ref{thm:M}),
which transforms our task into
checking the two assumptions of the theorem.

\subsection{Checking the cohomological assumption}
Our ultimate goal is to apply Thm.~\ref{thm:M},
and we thus need a way to establish the cohomological assumption,
namely,
that the canonical maps
$b^*R^is_* \to R^i\sigma_*B^*$, for $i=0,1$, are isomorphisms.
The idea is that a theory of enriched local systems usually lives inside some ambient theory of (perverse or constructible) sheaves,
where an analogue of the cohomological assumption holds.
A motivating example to have in mind is
the way admissible variations of mixed Hodge structures
sit in the larger category of mixed Hodge modules.

\begin{notation}\label{notation:C}
Let $\clg{D}(X)$ be the derived category $\clg{D}^\rom{b}_\rom{c}(X, F)$
of bounded complexes of sheaves of $F$-vector spaces on $X^\rom{an}$ with constructible cohomology.
\end{notation}

\begin{rmk}
Given a morphism $f\colon Y\to X$,
we write $Rf_*$ and $f^*$ for the pushforward and pullback on $\clg{D}(\blank)$.
Note that $f^*$ takes local systems to local systems,
and so restricts to a functor $f_{\Loc}^*\colon \Loc(X) \to \Loc(Y)$.
Moreover, when $f$ is such that the restriction of $R^0f_*$ to local systems lands in (ind-)local systems (e.g. when $X=\pt$),
then this restriction
gives the right adjoint $f^{\Loc}_*$ of $f_{\Loc}^*$.
\emph{However}, for $i>0$, the restriction of $R^if_*$ will generally \emph{not} agree with $R^i f^{\Loc}_*$,
even when it takes values in (ind-)local systems.
\end{rmk}

\begin{defn}\label{defn:ambient}
Let $\clg{E}$ be a weak theory of enriched local systems,
and $X\in\SmVar_k$,
with structure morphism $p\colon X \to \pt$.
An \emph{ambient theory} $\clg{A}(X)$ for $\clg{E}$ over $X$ is
a triangulated category $\clg{A}(X)$,
together with
\begin{enumerate}[(a)]
\item
a triangulated functor $\inc_\clg{A} \colon \Db(\clg{E}(X)) \to \clg{A}(X)$;
\item
a triangulated \emph{realization} functor $b^*_\clg{A} \colon \clg{A}(X)\to \clg{D}(X)$;
\item\label{enum:ambient_realization_iso}
an isomorphism $b^*_\clg{A} \circ \inc_\clg{A} \simeq \inc \circ\, b^*_\clg{E}$
(where $b^*_\clg{E}$ is viewed as a functor
$\Db(\clg{E}(X)) \to \Db(\Loc(X))$).
\end{enumerate}
Let $p^*_\clg{A} := \inc_\clg{A} \circ\, p^*_\clg{E}$,
where $p^*_\clg{E}$ is viewed as a functor
$\Db(\clg{E}(\pt)) \to \Db(\clg{E}(X))$.
We moreover require the following conditions to hold:
\begin{enumerate}[(i)]
\item\label{enum:ambient_adjoint}
the pullback functor $p^*_\clg{A}$ has a right adjoint
$p_*^\clg{A} \colon \clg{A}(X) \to \Db(\clg{E}(\pt))$;
\item\label{enum:ff_from_E_to_A}
the composition $\clg{E}(X) \to \Db(\clg{E}(X)) \to \clg{A}(X)$ is fully faithful;
\item\label{enum:E_after_real}
an object $M$ in $\clg{A}(X)$ lies in the full subcategory $\clg{E}(X)$ if and only if
$b_\clg{A}^*M$ lies in $\Loc(X)$;
\item\label{enum:ambient_real_comp}
the canonical map
$b^* p_*^\clg{A} \to Rp_*b^*_\clg{A}$
is an isomorphism;
\item\label{enum:ambient_truncated}
for every $M$ in $\clg{E}(X)$,
the canonical map $p_*^\clg{A}M \to \tau^{\geq 0}p_*^\clg{A}M$
is an isomorphism.
\end{enumerate}
\end{defn}

\begin{rmk}\label{rmk:ambient_canonical_maps}
We have two natural transformations:
\begin{align*}
Rp_*^{\clg{E}} \to p_*^{\clg{A}}\circ\inc_\clg{A}, \quad
b^* \circ\, p_*^{\clg{A}} \to Rp_* \circ\, b^*_{\clg{A}}.
\end{align*}
The former corresponds under adjunction to
\begin{align*}
p^*_\clg{A} \circ Rp_*^\clg{E}
  = \inc_\clg{A} \circ\, p^*_\clg{E} \circ Rp^*_\clg{E}
  \to \inc_\clg{A}.
\end{align*}
The latter is the ``canonical map'' in~\heqref{enum:ambient_real_comp}
and corresponds under adjunction to
\begin{align*}
p^* \circ\, b^* \circ\, p_*^{\clg{A}}
  &\simeq \inc \circ\, b^*_\clg{E} \circ\, p^*_\clg{E} \circ\, p_*^{\clg{A}} \\
&\xrightarrow{\heqref{enum:ambient_realization_iso}}
  b^*_\clg{A} \circ \inc_\clg{A} \circ\, p^*_\clg{E} \circ\, p_*^{\clg{A}} \\
&= b^*_\clg{A} \circ p^*_\clg{A} \circ\, p_*^{\clg{A}} \\
&\to b^*_\clg{A}.
\end{align*}
\end{rmk}
\begin{rmk}
Some care has to be taken with ind-categories at this point.
We have that $Rp_*^\clg{E}$ goes from $\Db(\Ind\clg{E}(X))$ to $\Db(\Ind\clg{E})$,
whereas $p_*^\clg{A} \circ \inc_\clg{A}$ goes between $\Db(\clg{E}(X))$ and $\Db(\clg{E})$.
The latter two categories are full triangulated subcategories of the former two
(see, e.g.,~\cite[Prop.~2.2]{huber_Ind}).
In particular,
when writing the first natural transformation in the previous remark,
we are suppressing the fact that we are actually extending the right-hand side functor to the ind-categories
(ind-category formation is functorial).
This is quite harmless and doesn't affect \nref{Prop.}{prop:pushf_assump_enriched_ambient},
for example.
\end{rmk}

\begin{exa}\
\begin{enumerate}
\item
The trivial example is given by $\clg{E} := \Loc(\blank)$ and $\clg{A}(X) := \clg{D}(X)$
(see \nref{Notation}{notation:C}).
\item
Another example of interest to us is given by $\clg{E} := \MHS(\blank)$
(admissible variations of mixed Hodge structure)
and $\clg{A}(X) = \clg{D}^\rom{b}(\MHM(X))$
(the bounded derived category of mixed Hodge modules).
See \nref{Lemma}{lem:MHM_ambient}.
\end{enumerate}
\end{exa}

For the rest of the section, we fix a weak theory $\clg{E}$ of enriched local systems,
$X\in\SmVar_k$,
and an ambient theory $\clg{A}(X)$ for $\clg{E}$ over $X$.

\begin{lem}\label{lem:E_extn_in_A}
The full subcategory $\clg{E}(X)$ is closed under extensions in $\clg{A}(X)$.
\end{lem}
\begin{proof}
Let $E$ be an extension in $\clg{A}(X)$ of
$M$ and $N$, both lying in the full subcategory $\clg{E}(X)$
(see \ref{defn:ambient}\heqref{enum:ff_from_E_to_A}).
Then $b^*E$ is an extension in $\clg{D}(X)$ of local systems $b^*M$ and $b^*N$,
so $b^*E$ is a local system.
We conclude using~\heqref{enum:E_after_real} in Def.~\ref{defn:ambient}.
\end{proof}

\begin{lem}\label{lem:ambient_ext_group_iso}
Given $M$ and $N$ in $\clg{E}(X)$,
the map
\begin{align*}
\Hom_{\Db(\clg{E}(X))}(M, N[1]) \to \Hom_{\clg{A}}(\inc(M), \inc(N)[1])
\end{align*}
induced by $\inc$ is bijective.
\end{lem}
\begin{proof}
Surjectivity follows immediately from \nref{Lemma}{lem:E_extn_in_A}.
For injectivity,
let $\delta$ be an element in the left-hand side.
Then $\delta$ and $\inc(\delta)$ correspond to extensions
\begin{align*}
1 \to N \to E \to M \to 1, \quad \inc(N) \to \inc(E) \to \inc(M) \to,
\end{align*}
in $\clg{E}(X)$ and $\clg{A}(X)$, respectively.
Suppose the latter splits,
i.e.,
$\inc(E) \to \inc(M)$ has a section.
By the fullness in \heqref{enum:ff_from_E_to_A},
we can lift this section back to $\clg{E}(X)$,
thereby splitting the extension corresponding to $\delta$,
and we're done.
\end{proof}

\begin{prop}[{Cf.~\cite[Lemma~4.3]{daddezio-esnault}}]\label{prop:pushf_assump_enriched_ambient}
The natural transformation
\begin{align*}
R^ip_*^{\clg{E}} \to H^ip_*^{\clg{A}} \circ \inc_\clg{A},
\end{align*}
is an isomorphism for $i=0,1$.
(See \nref{Rmk.}{rmk:ambient_canonical_maps}.)
\end{prop}
\begin{proof}
To reduce clutter,
let us suppress the functor $\inc_\clg{A}$ from the notation for the duration of the proof;
it can easily be reinserted.
We begin with the $i=0$ case.
Note that, naturally in $E\in\clg{E}$ and $M\in\clg{E}(X)$, we have
\begin{align*}
\Hom(E, p^\clg{E}_*M)
  = \Hom(p_\clg{E}^*E, M)
  = \Hom(p_\clg{A}^*E, M)
  &= \Hom(E, p^\clg{A}_* M) \\
  &= \Hom(E, H^0p^\clg{A}_*M),
\end{align*}
using
property~\heqref{enum:ff_from_E_to_A}
and property~\heqref{enum:ambient_truncated},
which implies that $p^\clg{E}_* \simeq  H^0p^\clg{A}_*$.

Next, we deal with the more interesting $i=1$ case.
Applying $\Hom(\blank,\blank)$ (in the second argument) to the morphism of triangles
\begin{equation*}\begin{tikzcd}
R^0p_*^{\clg{E}}M \ar[d] \ar[r] & Rp_*^{\clg{E}}M \ar[d] \ar[r] & \tau^{\geq 1}Rp_*^{\clg{E}}M \ar[d] \ar[r] & {} \\
H^0p_*^{\clg{A}}M \ar[r] & p_*^{\clg{A}}M \ar[r] & \tau^{\geq 1}p_*^{\clg{A}}M \ar[r] & {}
\end{tikzcd}\end{equation*}
for $M\in\clg{E}(X)$ (keeping property~\heqref{enum:ambient_truncated} in mind),
yields a morphism of (vertical) exact sequences
\begin{equation*}\begin{tikzcd}
\Hom(\blank, R^0p_*^{\clg{E}}M[1]) \ar[d] \ar[r]
  & \Hom(\blank, (H^0p_*^{\clg{A}}M)[1]) \ar[d] \\[-0.5em]
\Hom(\blank, Rp_*^{\clg{E}}M[1]) \ar[d] \ar[r]
  & \Hom(\blank, p_*^{\clg{A}}M[1]) \ar[d] \\[-0.5em]
\Hom(\blank, (\tau^{\geq 1}Rp_*^{\clg{E}}M)[1]) \ar[d] \ar[r]
  & \Hom(\blank, (\tau^{\geq 1}p_*^{\clg{A}}M)[1]) \ar[d] \\[-0.5em]
\Hom(\blank, R^0p_*^{\clg{E}}M[2]) \ar[r]
  & \Hom(\blank, (H^0p_*^{\clg{A}}M)[2])
\end{tikzcd}\end{equation*}
Denote the regular representation $b_*\bm1$ in $\Ind\clg{E}$ by $\clg{O}$.
After evaluating these long exact sequences at the pro-object $\clg{O}^\vee$,
the topmost and bottommost terms vanish,
since $\clg{O}^\vee$ doesn't have any non-trivial extensions
(cf. \nref{Prop.}{prop:split-reg-rep}).
This lack of extensions implies,
in fact,
that there are no morphisms from $\clg{O}^\vee$ to $\clg{D}^{\leq -1}(\clg{E})$,
and we thus have a commutative diagram
\begin{equation*}\begin{tikzcd}
\Hom(p_\clg{E}^*\clg{O}^\vee, M[1]) \ar[r] \ar[d, equals]
  & \Hom(p_\clg{A}^*\clg{O}^\vee, M[1]) \ar[d, equals] \\[-0.8em]
\Hom(\clg{O}^\vee, Rp^\clg{E}_*M[1]) \ar[r] \ar[d, "\sim"]
  & \Hom(\clg{O}^\vee,p^\clg{A}_*M[1]) \ar[d, "\sim"] \\[-0.5]
\Hom(\clg{O}^\vee, (\tau^{\geq 1}Rp^\clg{E}_*M)[1]) \ar[r] \ar[d, equals]
  & \Hom(\clg{O}^\vee, (\tau^{\geq 1}p^\clg{A}_*M)[1]) \ar[d, equals] \\[-0.8em]
\Hom(\clg{O}^\vee, R^1p^\clg{E}_*M) \ar[r]
  & \Hom(\clg{O}^\vee, H^1p^\clg{A}_*M)
\end{tikzcd}\end{equation*}
Using the definition of $p^*_\clg{A}$ and \nref{Lemma}{lem:ambient_ext_group_iso},
the leftmost vertical arrow is an isomorphism.
Thus, applying
$\Hom(\clg{O}^\vee, \blank)$
to the morphism
$R^1p^\clg{E}_*M \to H^1p^\clg{A}_*M$,
yields an isomorphism.
We conclude by the conservativity of
$\Hom(\clg{O}^\vee, \blank) = \Hom(\blank, \clg{O}) \circ (\blank)^\vee \simeq b^*$.
\end{proof}

\begin{rmk}
In the above proof,
we are implicitly viewing categories as full subcategories of their respective pro-categories
in a number of places after the introduction of the pro-object $\clg{O}^\vee$.
\end{rmk}

Here, finally,
is the point of this subsection:
\begin{prop}\label{prop:enriched_coh_assump}
Let $X$ be a variety in $\SmVar_k$
with a $k$-point $x$.
If $\clg{E}$ admits an ambient theory over $X$,
then the cohomological assumption is satisfied,
i.e.,
\begin{equation*}
b^*R^ip_*^\clg{E} \to R^ip_*b^*_\clg{E}
\end{equation*}
is an isomorphism for $i=0,1$.
\end{prop}
\begin{proof}
This follows from property~\heqref{enum:ambient_real_comp},
\nref{Prop.}{prop:pushf_assump_enriched_ambient},
and the commutativity of the following diagram:
\begin{equation*}\begin{tikzcd}
b^*Rp_*^{\clg{E}} \ar[d] \ar[r] & Rp_*b^*_{\clg{E}} \\
b^*p_*^{\clg{A}}\inc_{\clg{A}} \ar[r] & Rp_*b^*_{\clg{A}}\inc_{\clg{A}} \ar[u, "\heqref{enum:ambient_realization_iso}"']
\end{tikzcd}\end{equation*}
(see \nref{Rmk.}{rmk:ambient_canonical_maps}).
\end{proof}

\subsection{Checking the surjectivity assumption}
We begin with a definition:
\begin{defn}\label{defn:enriched_ss}
Say a weak theory $\clg{E}$ of enriched local systems
\emph{respects semi-simplicity}
if for every $X$ in $\SmVar_k$,
we have that
$b^* \colon \clg{E}(X) \to \Loc(X)$
sends semisimple objects to semisimple objects.
\end{defn}

It turns out that this property is all we need for the surjectivity assumption:
\begin{prop}\label{prop:enriched_surjectivity}
Assume the characteristic of $F$ is zero.
Suppose $\clg{E}$ respects semi-simplicity.
Then, for any $X$ in $\SmVar_k$ admitting an ambient theory
and any $k$-point $x$,
the map
$\bar\varrho \colon \pi_1(X,x) \to S(X,x)$
has Zariski dense image.
\end{prop}
\begin{proof}
We show it on the Tannakian side,
i.e., we work with
\begin{equation*}
\bar\varrho^* \colon \Rep S(X,x) \to \Loc(X).
\end{equation*}
By \nref{Prop.}{prop:tannaka-mor-cor},
we need to prove that the image of $\bar\varrho^*$ is stable under taking subobjects in $\Loc(X)$.
Recall that we have a commutative diagram
\begin{equation*}\begin{tikzcd}
\clg{E}(X) \ar[d, "\iota^*" '] \ar[r, "b^*"] & \Loc(X) \\
\Rep K(X,x) \ar[ur, "\varrho^*"] & \ar[l, hook'] \Rep S(X,x) \ar[u, "\bar\varrho^*" ']
.
\end{tikzcd}\end{equation*}

\emph{Step 1.}
Let $\clg{E}(X)_\rom{ss}$ denote the full subcategory of semisimple objects in $\clg{E}(X)$.
We note that under $\iota^*$, $\clg{E}(X)_\rom{ss}$ lands in the full subcategory
$\Rep S(X,x) \subseteq \Rep K(X,x)$.
Let $G(X,x)_\rom{red}$ be the Tannaka dual of $\clg{E}(X)_\rom{ss}$.
It's the maximal pro-reductive quotient of $G(X,x)$, as $F$ is of characteristic zero.
As the pro-unipotent radical $R_\rom{u}K(X,x)$ is a characteristic subgroup of $K(X,x)$,
it is normal in the larger group $G(X,x)$.
Since $R_\rom{u}G(X,x)$ is the largest connected, pro-unipotent, normal subgroup of $G(X,x)$,
it must contain $R_\rom{u}K(X,x)$.
Thus, $K(X,x) \to G(X,x)_\rom{red}$ factors uniquely through $K(X,x) \to S(X,x)$.

\emph{Step 2.}
Next,
we note that the map $S(X,x) \to G(X,x)_\rom{red}$ is injective.
Indeed, this follows from
$K(X,x) \cap R_\rom{u}G(X,x) = R_\rom{u}K(X,x)$,
which holds true by the previous step and the definition of the right-hand side.
In particular, for every object $A \in \Rep S(X,x)$,
there exists $M \in \clg{E}(X)_\rom{ss}$
such that $A$ is a subquotient of $\iota^*M$, in $\Rep S(X,x)$.
By the semi-simplicity of $\Rep S(X,x)$,
$A$ can actually be taken to be a subobject of $\iota^*M$.

\emph{Step 3.} We now show that
$\bar\varrho^*$ is full.
Take two objects $M$ and $N$ in $\clg{E}(X)$.
Then,
writing $\iota^*\colon \clg{E}(X) \to \Rep K$, as before,
we have that
\begin{align*}
\Hom(\iota^* M, \iota^* N) = b^*s_*\iHom(M,N)
&= \sigma_*\iHom(b_X^*M, b_X^*N) \\
&= \Hom(\varrho^*\iota^*M, \varrho^*\iota^* N),
\end{align*}
using the $i=0$ part of \nref{Prop.}{prop:pushf_assump_enriched_ambient};
so we have fullness (for $\varrho^*$ and hence also $\bar\varrho^*$) on objects coming from $\clg{E}(X)$.

By the previous step,
every object in $\Rep S(X,x)$ is a direct summand of an object coming from $\clg{E}(X)_\rom{ss}$.
Let two $S(X,x)$-representations $A$ and $B$ be direct summands of $\iota^*M$ and $\iota^*N$,
and cut out by projectors $e$ and $f$, respectively.
Then,
\begin{align*}
\Hom(A, B) = f \circ \Hom(\iota^*M, \iota^*N) \circ e
= \varrho^*(f) \circ \Hom(\varrho^*\iota^*M, \varrho^*\iota^*N) \circ \varrho^*(e) \\
= \Hom(\varrho^*A, \varrho^*B).
\end{align*}

Note finally, that the fullness of $\bar\varrho^*$
implies that its image is stable under taking direct summands
(as fullness lets us lift projectors).

\emph{Step 4.}
Now we are ready to conclude.
Let $A \in \Rep S(X,x)$ and $L \subseteq \varrho^*A$.
By Step 2,
there is $M \in \clg{E}(X)_\rom{ss}$
such that $A \subseteq \iota^*M$.
Thus, $L \subseteq \bar\varrho^*\iota^*M = b^*M$,
whence $L$ is a direct factor of the right-hand side, as $\clg{E}$ respects semi-simplicity,
and we are done by Step 3.
\end{proof}

\subsection{A corollary to Malcev completeness}
Fix an enriched local system $M$ in $\clg{E}(X)$
on some $X\in\SmVar_k$ with a rational point $x$.
Denote by $M_x$ the vector space $x^*b^* M$.
We may consider the associated monodromy representation $\pi_1(X,x) \to \GL(M_x)$.
Take the Malcev completion of this morphism,
and denote it by $\hat\pi_1^M(X, x)$.

\begin{prop}\label{prop:generalized_hain_enriched}
Suppose that $\pi_1(X,x) \to K(X,x)$ is Malcev complete.
Then the group $\hat\pi_1^M(X, x)$ is canonically enriched in $\clg{E}(\pt)$,
in the sense that $\clg{O}(\hat\pi_1^M(X, x))$ is naturally a Hopf algebra object in $\Ind\clg{E}(\pt)$,
or equivalently that there is a natural action by $G(k)$ on $\hat\pi_1^M(X,x)$.
\end{prop}
\begin{proof}
Consider the full subcategory $\langle M, \clg{E}(\pt)\rangle^\otimes$ tensor-generated by $M$ and $\clg{E}(\pt)$ in $\clg{E}(X)$,
i.e., consisting of subquotients of sums of tensor products of $M$, $M^\vee$ and constant objects
(i.e., objects coming from $\clg{E}$).
We obtain a commutative diagram
\begin{equation*}\begin{tikzcd}
1\ar[r]
  & \scr{G}(\langle \iota^*M\rangle^\otimes) \ar[r]
  & \scr{G}(\langle M, \clg{E}(\pt)\rangle^\otimes) \ar[r]
  & G(k) \ar[r]
  & 1
  \\
1\ar[r]
  & \pi_1^\clg{E}(X,x) \ar[u, two heads] \ar[r]
  & G(X,x) \ar[u, two heads] \ar[r]
  & G(k) \ar[u, equal] \ar[r]
  & 1
.
\end{tikzcd}\end{equation*}
We have seen that the bottom row is (split) exact,
and the (split) exactness of the top row follows.

We have a factorization
\begin{align*}
\pi_1(X,x)\to \scr{G}(\langle \iota^*M\rangle^\otimes) \hookrightarrow \GL(M_x),
\end{align*}
of the monodromy representation.
By \nref{Lemma}{lem:alt_malcev_defn},
the Malcev completion of $\pi_1(X,x) \to \GL(M_x)$
only depends on the image in $\GL(M_x)$,
which by this factorization lies in the subgroup
$\scr{G}(\langle \iota^*M\rangle^\otimes)$.
Thus, $\hat\pi_1^M(X, x)$ is the Malcev completion of
$\pi_1(X,x)\to \scr{G}(\langle \iota^*M\rangle^\otimes)$.

Next, $\hat\pi_1^M(X, x)$ can be obtained as the Malcev completion of
$\pi_1^\clg{E}(X,x)\to \scr{G}(\langle\iota^*M\rangle^\otimes)$,
by \nref{Lemma}{lem:malcev_malcev}.
Note that $G(k)$ acts on this morphism,
and by functoriality of Malcev completions (\nref{Prop.}{prop:malcev_functoriality}),
it acts on the Malcev completion, $\hat\pi_1^M(X, x)$.
This finishes the proof.
\end{proof}

\section{Hodge theory and motives}\label{sec:applications}
In this section,
we present our two main applications of the Malcev completeness criterion (Thm.~\ref{thm:M}).
By the previous section,
for each choice of theory of enriched local systems,
we simply need to check the existence of an ambient theory (Def.~\ref{defn:ambient})
and that semi-simplicity is satisfied (Def.~\ref{defn:enriched_ss}).
For the entirety of this section,
the field of coefficients $F$ will be $\Q$.
We extensively use the notation and terminology introduced in the previous section.

\subsection{Variations of Hodge structures}\label{subsec:hodge}
We start with weak theory of enriched local systems
$\clg{E}(\blank) = \MHS(\blank)$
given by \emph{admissible variations of mixed Hodge structures}.
(See, for example,~\cite[Ch.~10,14]{peters-steenbrink}.)
Let $X \in \SmVar_k$ and let $x \in X(\C)$.\footnote{%
  In this setting, $x$ doesn't need to be rational,
  as $\clg{E}(\Spec k) = \clg{E}(\Spec \C)$.%
}
In this context,
we denote the corresponding fundamental group
(previously $\pi_1^{\clg{E}}(X, x)$)
by $\pi_1^\rom{Hdg}(X, x)$
and call it the \emph{Hodge fundamental group}.
The representations of the Hodge fundamental group
are called \emph{local systems of Hodge origin},
and form a full subcategory of the category $\Loc(X)$ of local systems,
see Def.~\ref{defn:E-fund}.
They are exactly those local systems that arise as subquotients
of local systems underlying admissible variations of mixed Hodge structures.

Over the point, we get the category $\MHS$ of graded-polarisable mixed Hodge structures.
We call the Tannaka dual group $H$ of $\MHS$ the \emph{Hodge group}.

\begin{lem}\label{lem:MHM_ambient}
We have that $\MHS(\blank)$ is a weak theory of enriched local systems,
and that
(the derived category of) mixed Hodge modules
provides an ambient theory in the sense of \nref{Def.}{defn:ambient}.
\end{lem}
\begin{proof}
Let $X \in \SmVar_k$.
We write $\MHM(X)$ for the category of mixed Hodge modules on $X$,
and $\clg{A}(X)$ for the bounded derived category $\Db(\MHM(X))$.
The functor $\inc_\clg{A}$ is the composition of the usual inclusion $\MHS(X) \subseteq \MHM(X)$,
followed by the inclusion of the heart $\MHM(X) \subseteq \clg{A}(X)$,
and then a shift by $-\dim(X)$.
(This is because local systems embed via $\clg{D}(X)$ into $\scr{P}(X)$ after a shift by $\dim(X)$.)
Condition~\heqref{enum:ff_from_E_to_A} is then clear,
and condition~\heqref{enum:ambient_adjoint} is in~\cite[Thm.~0.1]{MHM}.
For the realization functor $b^*_\clg{A}$,
we consider the (exact) forgetful functor
$\MHM(X) \to \scr{P}(X)$
from mixed Hodge modules to perverse sheaves.
Deriving it, and combining with Beilinson's equivalence
$\Db(\scr{P}(X)) \simeq \clg{D}(X)$ (\cite[Thm.~1.3]{beilinson}),
we get the functor $b^*_\clg{A} \colon \clg{A}(X) \to \clg{D}(X)$
(denoted by $\rom{rat}$ in \cite{MHM}).
The isomorphism in~\heqref{enum:ambient_realization_iso} follows.

Condition~\heqref{enum:E_after_real} follows from~\cite[Thm.~0.2]{MHM}
and the definition of smooth mixed Hodge modules that precedes it.
Condition~\heqref{enum:ambient_real_comp} is part of~\cite[Thm.~0.1]{MHM}.
Finally,
in this case,
condition~\heqref{enum:ambient_truncated} can be checked after realization,
where it follows from the fact that $Rp_*$ has perverse cohomological amplitude $\geq -\dim(X)$
(\cite[{}4.2.4]{BBD}).
\end{proof}

Note that $\clg{E}(\blank) = \MHS(\blank)$
respects semi-simplicity
by Deligne's semi-simplicity theorem, \cite[Thm.~4.2.6]{deligne-hodge-II}.
Thus, plugging Propositions~\ref{prop:enriched_coh_assump} and~\ref{prop:enriched_surjectivity}
into Thm.~\ref{thm:M}
and combining with Prop.~\ref{prop:generalized_hain_enriched},
we deduce the following:
\begin{thm}\label{thm:generalized_hain}
Let $X$ be a smooth, geometrically connected $k$-variety and $x \in X(\C)$ a complex point.
\begin{enumerate}
\item
The map $\pi_1(X,x)\to K^\rom{Hdg}(X, x)$
is Malcev complete.
Equivalently,
$\pi_1^\rom{Hdg}(X,x) = K^\rom{Hdg}(X,x)$, and
the full subcategory of local systems on $X$ of Hodge origin
is stable under extension in the category of local systems on $X$.
In particular,
every unipotent local system on $X$ is of Hodge origin.
\item
For $V$ an admissible variation of mixed Hodge structures,
the Malcev completion $\hat\pi_1^V(X,x)$
of the associated monodromy representation
$\pi_1(X,x) \to \GL(V_x)$
is canonically equipped with a mixed Hodge structure,
in the sense that it admits a canonical action by the Hodge group.
\qed
\end{enumerate}
\end{thm}

\begin{rmk}
Part 1 of \nref{Thm.}{thm:generalized_hain} has been previously shown by
D'Addezio and Esnault in~\cite[Thm.~4.4]{daddezio-esnault},
by a different method.
Part 2 is a generalization of~\cite[Thm.~13.1]{hain},
where $V$ is assumed to be pure and polarizable,
and $\pi_1(X,x)$ is assumed to map Zariski densely onto $\Aut(V_x,\langle,\rangle)$.
In particular, we prove Conjecture~5.5 from~\cite{arapura-hodge}.
\end{rmk}

\subsection{Motivic local systems}
Ivorra and Morel have constructed and studied categories of \emph{perverse (Nori) motives} in~\cite{ivorra-morel},
and we will briefly review their construction.
Given an additive category $\clg{Q}$
and an additive functor $T\colon \clg{Q} \to \clg{A}$
to an abelian category $\clg{A}$,
they construct a universal abelian category $\mathbf{A}^\rom{ad}(\clg{Q},T)$
factoring this functor along an exact and faithful functor to $\clg{A}$.
Below, this universal category $\mathbf{A}^\rom{ad}(\clg{Q},T)$ will be called
the \emph{Nori category} of $T \colon \clg{Q} \to \clg{A}$.

Given a $k$-variety (separated and of finite type) $X$,
they consider the additive functor
\begin{align*}
{}^pH^0 \circ \Bti^* \colon \DA_\rom{ct}(X) \to \clg{P}(X),
\end{align*}
from the (triangulated) category of constructible \'etale motivic sheaves (with $\Q$-coefficients) on $X$
to ($\Q$-linear) perverse sheaves on $X$.
We denote the Nori category of this functor by $\PM(X)$:
it is the category of \emph{perverse motives} considered in~\cite{ivorra-morel}.\footnote{%
  They denote it by $\scr{M}(X)$.}
Here, $\Bti^*$ is the Betti realization functor constructed by Ayoub in~\cite{ayoub_betti}.
We note that the category $\PM(k)$ of perverse motives over the field $k$
is equivalent to the classical category $\NM(k)$ of Nori motives
(\cite[Prop.~2.11]{ivorra-morel}).

Note that $\NM(k)$ is neutral Tannakian,
and that its Tannaka dual $\scr{G}^\rom{mot}(k)$ is the \emph{motivic Galois group} of $k$.
It's a pro-algebraic group (over $\Q$).

Perverse motives come equipped with a realization functor $\PM(X) \to \clg{P}(X)$ to perverse sheaves
which we denote by $b_X^*$.
Ivorra and Morel develop a four-functor formalism for perverse motives.
In~\cite{luca},
this is extended to a full six-functor formalism,
with the construction of a tensor product on $\Db\PM(X)$.

\begin{defn}\label{defn:mot_loc_sys}
Let $X$ be a smooth, geometrically connected $k$-variety.
Denote the full subcategory of $\PM(X)$
on objects that are sent into $\Loc(X)[d]$ under $b_X^*$
by $\MLS(X)$.
It's the category of \emph{motivic local systems} on $X$.
\end{defn}

By~\cite[Thm.~6.2]{luca},
these are neutral Tannakian categories over $\Q$,
neutralized by $x^* \circ b_X^*$
for any complex point $x \in X(\C)$.
Since pullbacks preserve local systems,
the same is true for motivic local systems.
Moreover, the Betti realization induces faithful exact functors
$b^* \colon \MLS(X) \to \Loc(X)$.
In conclusion,
we have a weak theory of enriched local systems
$\clg{E} = \MLS$
(Def.~\ref{defn:enriched}).

We denote the associated categories $\Loc_\clg{E}(X)$ (see Def.~\ref{defn:E-fund})
by $\Loc_\rom{geo}(X)$
and call its objects local systems \emph{of geometric origin}.
Their dual groups we denote by $\pi_1^\rom{mot}(X,x)$
(for a complex point $x$).
We moreover denote
the Tannaka dual of $\MLS(X)$ relative the fibre functor $x^*\circ\Bti^*$ by $G(X,x)$.
As in Section~\ref{sec:enriched},
we denote the kernel of $G(X,x) \to \scr{G}^\rom{mot}(k)$ by $K(X,x)$,
and its maximal pro-reductive quotient by $S(X,x)$.
Finally,
note that if $x$ is a rational point,
then we get a section of the homomorphism $G(X,x) \to \scr{G}^\rom{mot}(k)$.

In summary, for $X$ a geometrically connected, smooth $k$-variety
with a $k$-point $x$, we have a commutative diagram
\begin{equation*}\begin{tikzcd}
&[-1em] S(X,x) \\[-0.5em]
1 \ar[r] & K(X,x) \ar[u, two heads] \ar[r]
  &[-1em] G(X,x) \ar[r]
  & \scr{G}^\rom{mot}(k) \ar[l, bend right, start anchor = north west, end anchor = north east, yshift=-.5em] \ar[r] & 1 \\[-0.5em]
\pi_1(X,x) \ar[ur, bend right=4] \ar[r, two heads] & \pi_1^\rom{mot}(X, x) \ar[u, hook]
\end{tikzcd}\end{equation*}
whose middle row is (split) exact.
(Just as in Section~\ref{sec:enriched}.)

\begin{lem}\label{lem:motivic_ambient}
Let $X \in \SmVar_k$.
Then $\Db(\PM(X))$ is an ambient theory for $\MLS$ over $X$ (Def.~\ref{defn:ambient}).
\end{lem}
\begin{proof}
We follow the notation in Def.~\ref{defn:ambient} and
let $\mathcal{A}(X) := D^\mathrm{b}(\PM(X))$.
For $\inc_{\mathcal{A}}$,
we take the
composition $\MLS(X) \hookrightarrow \PM(X) \hookrightarrow \Db(\PM(X))$.
For $b^*_\mathcal{A}$,
we take the Betti realization.
Lastly,
the isomorphism asked for in \heqref{enum:ambient_realization_iso} is actually an equality in this case.
We now quickly check that the conditions for an ambient theory hold in this case.

Condition~\eqref{enum:ambient_adjoint} is true by~\cite{ivorra-morel},
and conditions~\eqref{enum:ff_from_E_to_A} and~\eqref{enum:E_after_real} are true by definition.
Next,~\eqref{enum:ambient_real_comp} is true by~\cite{ivorra-morel}
(compatibility of the four operations with the Betti realization).
Finally,
condition~\eqref{enum:ambient_truncated} can
(by conservativity)
be checked under realization,
where it follows from~\cite[4.2.4]{BBD}.
\end{proof}

In~\cite[Section 6]{ivorra-morel},
it is shown that their perverse motives admit a robust theory of weights.
This lets us show the following:
\begin{lem}
The theory $\clg{E}(\blank) = \MLS(\blank)$ respects semi-simplicity
(Def.~\ref{defn:enriched_ss}).
\end{lem}
\begin{proof}
We need to show that
$b^* \colon \MLS(X) \to \Loc(X)$ sends semisimple objects to semisimple objects.
But thanks to \cite[Thm.~0.1]{swann},
perverse motives admit a Hodge realization,
and
we can factor the Betti realization through the Hodge realization as
\begin{equation*}
\MLS(X) \to \MHS(X) \to \Loc(X).
\end{equation*}
Moreover,
in the proof of~\cite[Prop.~4.7]{swann}
it is explained how~\cite[Cor.~6.27]{ivorra-morel}
implies that the Hodge realization sends pure objects to pure objects.
In both categories,
the semisimple objects are exactly the ones obtained as direct sums of pure ones.
We can thus conclude using Deligne's semi-simplicity theorem.
\end{proof}

Thus,
as in the Hodge theoretic case,
we get the following theorem:
\begin{thm}\label{thm:motivic_hain}
Let $X \in \SmVar_k$ with a rational point $x$.
\begin{enumerate}
\item
The homomorphism $\pi_1(X,x) \to K(X,x)$ is Malcev complete.
Equivalently,
\begin{equation*}
\pi_1^\rom{mot}(X,x) = K(X,x) := \ker(G(X,x) \to \scr{G}^\rom{mot}(k)),
\end{equation*}
and
the full subcategory of local systems on $X$ of geometric origin is stable under extension.
In particular,
every unipotent local system on $X$ is of geometric origin.
\item
For a motivic local system $M$ on $X$,
the Malcev completion $\hat\pi_1^M(X,x)$
of the associated monodromy representation
$\pi_1(X,x) \to \GL(M_x)$,
is motivic,
in the sense that it carries a canonical action by the motivic Galois group.
\qed
\end{enumerate}
\end{thm}
The second part is a (Nori) motivic refinement
of Hain's theorem.

\section{De Rham fundamental groups}\label{sec:deRham}
We round out the paper with a slightly different kind of application.
The goal of this section is to use
the unipotent isomorphism criterion (Thm.~\ref{thm:U2})
to give a quick proof of Lazda's theorem.
The theorem says that,
for a nice enough family,
the relative unipotent de Rham fundamental group of the family
and the unipotent fundamental group of the special fibre
are isomorphic.

Let $k=F$ be a field of characteristic zero,
let $S$ be a smooth $k$-scheme,
let $f\colon X\to S$ be a smooth $S$-scheme with an $S$-point $x\colon S \to X$,
let $s\colon \pt\to S$ be a closed point of $S$,
and write $x$ also for the closed point of $X$ given by $x \circ s$.
Denote the fibre of $X$ over $s$ by $X_s$,
and the canonical inclusion by $i\colon X_s\to X$.
For $Y$ a smooth $k$-scheme,
let $\IC(Y)$ denote the category of vector bundles on $Y$ with regular integrable connection:
in short, the category of \emph{regular connections}
(see for example \cite[Def.~1.2.2, Def.~5.3.2]{Dmod}).
If $Y$ is moreover geometrically connected,
then $\IC(Y)$ is a $k$-linear Tannakian category.

Given a morphism of smooth $k$-schemes $g\colon Y\to Z$,
we have an induced functor $g^*\colon\IC(Z) \to \IC(Y)$.
And in particular, if $y\colon \pt \to Y$ is a closed point of $Y$,
and $Y$ is geometrically connected,
then $y^*\colon\IC(Y)\to\VecOp$ is a neutral fibre functor.
In this case,
we denote the Tannaka dual of $\IC(Y)$ by $\pi_1^\rom{dR}(Y,y)$,
and call it the \emph{de Rham fundamental group} of $(Y,y)$.

Suppose that both $Y$ and $Z$ as above are geometrically connected and $g\colon Y \to Z$,
then,
following the terminology of \nref{Section}{sec:unipotency},
we denote the full subcategory of $g^*$-unipotent objects in $\IC(Y)$ by $\IC^\rom{un}_g(Y)$.
If $p\colon Y\to \pt$ is the structure map of $Y$,
then we simply write $\IC^\rom{un}(Y) := \IC^\rom{un}_p(Y)$:
this is the full subcategory of (absolutely) unipotent objects in $\IC(Y)$.
We denote the Tannaka dual of $\IC^\rom{un}(Y)$ by $\pi_1^\rom{udR}(Y,y)$,
and call it the \emph{unipotent de Rham fundamental group} of $(Y,y)$
(it is the unipotent completion of $\pi_1^\rom{dR}(Y,y)$).
We don't have a good symbol for the Tannaka dual of $\IC^\rom{un}_g(Y)$,
and we simply denote it by $\scr{G}(\IC^\rom{un}_g(Y),y)$,
following our general conventions.

\begin{rmk}
In the literature,
it is often the case that
only the unipotent de Rham fundamental group
is considered.
Usually, the word ``unipotent'' is then omitted from the terminology.
\end{rmk}

Assuming that $X$ and $S$ are geometrically connected,
and that $f$ has geometrically connected fibres,
we obtain a diagram of exact and faithful functors
\begin{equation*}\begin{tikzcd}
\IC^\rom{un}_f(X) \ar[d, "x^*"'] \ar[r, "i^*"] & \IC^\rom{un}(X_s) \ar[d, "x^*"'] \\
\IC(S) \ar[u, "f^*"', bend right] \ar[r, "s^*"] & \VecOp \ar[u, "f_s^*"', bend right]
\end{tikzcd}\end{equation*}
which is an instance of the Standard Situation (\nref{Section}{sec:standard}).
We get, in particular, a split exact sequence of pro-algebraic groups
\begin{equation*}\begin{tikzcd}
1
  \ar[r] &[-2.8em] K
  \ar[r] &[-1.8em] \scr{G}(\IC^\rom{un}_f(X),x)
  \ar[r] & \pi_1^\rom{dR}(S,s) \ar[l, bend right]
  \ar[r] &[-1em] 1 \\[-1em]
& \pi_1^\rom{udR}(X_s,x) \ar[u]
\end{tikzcd}\end{equation*}
and a morphism from $\pi_1^\rom{udR}(X_s,x)$ to the kernel $K$.
We call the kernel the \emph{relative unipotent de Rham fundamental group} of $X/S$
and denote it by $\pi_1^\rom{udR}(X/S,x)$.
If we could apply Thm.~\ref{thm:U2} to this situation,
we would conclude that $\pi_1^\rom{udR}(X_s,x) \to \pi_1^\rom{udR}(X/S,x)$ is an isomorphism.

In order to apply Thm.~\ref{thm:U2},
we have two conditions that need to hold:
the natural maps
\begin{align*}
s^* \circ R^kf_* \to R^kf_{s,*} \circ i^*
\end{align*}
are isomorphisms, for $k=0,1$
(we use subscript-asterisks to denote ind-right adjoints, as usual).

\begin{prop}
Assume that $k$ is algebraically closed.
Let $f \colon X \to S$ be a smooth family
with geometrically connected fibres,
let $s \colon S \to X$ be a section,
and keep all the notation introduced above.
Assume moreover that $f$ admits a good compactification.\footnote{%
  I.e., it factors as an open immersion $X \to \bar{X}$,
  followed by a smooth and proper map $\bar{X} \to S$,
  such that the complement of $X$ in $\bar{X}$
  is a relative normal crossing divisor.}
Then the natural maps
\begin{align*}
s^* \circ R^kf_* \to R^kf_{s,*} \circ i^*
\end{align*}
are isomorphisms for $k=0,1$.
\end{prop}
\begin{proof}
Following Lazda,
we set $f_*^\rom{dR} = \scr{H}^{-d}(f_+(\blank))$,
where $d$ is the relative dimension of $f\colon X \to S$,
and $f_+$ is the pushforward of regular holonomic complexes of $\scr{D}$-modules.
By~\cite[Lemma~1.9]{lazda},
$f_*^\rom{dR}$ takes $f$-unipotent regular connections to regular connections,
and,
viewing it as a functor
$\IC^\rom{un}_f(X) \to \IC(S)$,
it is equal to $f_*$.
Moreover,
pullback ($f^+$) of regular holonomic complexes becomes $f^*[-d]$
when restricted to the abelian category of regular connections.

Next, note that regular connections are stable under extension
in the category of regular holonomic $\scr{D}$-modules.
For the $k=1$ part,
the same strategy as in the proof of \nref{Prop.}{prop:pushf_assump_enriched_ambient} works,
and the details are left to the reader.
The input required from the theory of $\scr{D}$-modules to make the argument work
is that base change holds:
$f_{s,+}i^+ \simeq s^+f_+$ (\cite[Thm.~1.7.3, Rmk.~1.7.6]{Dmod}).
\end{proof}

Combining this result and~Thm.~\ref{thm:U2},
we get the following:
\begin{thm}[{\cite[Cor.~1.20]{lazda}}]
Under the assumptions of the previous proposition,
the map $\pi_1^\rom{udR}(X_s,x) \to \pi_1^\rom{udR}(X/S,x)$ is an isomorphism.
\end{thm}

\begin{rmk}
The assumption that $k$ is algebraically closed is not expected to be necessary
(cf.~\cite[Rmk.~1.21]{lazda}),
and the only place we use it is in the invocation of~\cite[Lemma~1.9]{lazda}.
\end{rmk}

\begin{rmk}
An anonymous referee produced an example showing
that if the family $f\colon X \to S$ fails to admit a section,
then the two groups in the theorem need not be isomorphic.
\end{rmk}

In work in progress,
I am applying the Malcev completeness criterion (Thm.~\ref{thm:M})
to this setting of de Rham fundamental groups.
It yields,
under appropriate hypotheses,
a result on the stability under extension
of regular connections over $X_s$ that come from regular connections over $X$.

\printbibliography

\end{document}